\documentclass[11pt,reqno]{amsart}
\usepackage{enumerate}
\usepackage{enumitem}
\usepackage{mathrsfs}
\usepackage{url}
\usepackage{cite}
\usepackage{comment}
\usepackage{graphicx}
\usepackage{upgreek}
\usepackage{tikz}
\usetikzlibrary{arrows,shapes,trees,backgrounds}
\usepackage{fancyhdr}
\usepackage{amsfonts, amsmath, amssymb, amscd, amsthm, bm, cancel}
\usepackage[numbers,sort&compress]{natbib}
\usepackage[%backref=page,
linktocpage=true,colorlinks,citecolor=magenta,linkcolor=blue,urlcolor=magenta]{hyperref}
\usepackage[margin=1.1in]{geometry}

\newtheorem{proposition}{Proposition}[section]
\newtheorem{theorem}[proposition]{Theorem}
\newtheorem{lemma}[proposition]{Lemma}

\theoremstyle{definition}
\newtheorem*{ack}{Acknowledgements}
\theoremstyle{remark}
\newtheorem{remark}[proposition]{Remark}

\numberwithin{equation}{section}
\allowdisplaybreaks

\begin{document}

\title[Geometric inequalities in hyperbolic space]{Geometric inequalities and their stabilities for modified quermassintegrals in hyperbolic space}

\author[C. Gao]{Chaoqun Gao}
\address{School of Mathematical Sciences, University of Science and Technology of China, Hefei 230026, P.R. China}
\email{\href{mailto:gaochaoqun@mail.ustc.edu.cn}{gaochaoqun@mail.ustc.edu.cn}}
%\author[Y. Wei]{Yong Wei}
%\address{School of Mathematical Sciences, University of Science and Technology of China, Hefei 230026, P.R. China}
%\email{\href{mailto:yongwei@ustc.edu.cn}{yongwei@ustc.edu.cn}}
\author[R. Zhou]{Rong Zhou}
%\address{School of Mathematical Sciences, University of Science and Technology of China, Hefei 230026, P.R. China}
\email{\href{mailto:zhourong@mail.ustc.edu.cn}{zhourong@mail.ustc.edu.cn}}

\date{\today}
%\thanks{}

\begin{abstract}
In this paper, we first consider the curve case of Hu-Li-Wei's flow for shifted principal curvatures of h-convex hypersurfaces in $\mathbb{H}^{n+1}$ proposed in \cite{Locallyconstrained}. We prove that if the initial closed curve is smooth and strictly h-convex, then the solution exists for all time and preserves strict h-convexity along the flow. Moreover, the evolving curve converges smoothly and exponentially to a geodesic circle centered at the origin. The key ingredient in our proof is the Heintze-Karcher type inequality for h-convex curves proved recently in \cite{Li2022Hyperbolic}.

As an application, we then provide a new proof of geometric inequalities involving weighted curvature integrals and modified quermassintegrals for h-convex curves in $\mathbb{H}^2$. We finally discuss the stability of these inequalities as well as Alexandrov-Fenchel type inequalities for modified quermassintegrals for strictly h-convex domains in $\mathbb{H}^{n+1}$. %This is motivated by the work \cite{KWONG2023109213}, where Kwong and Wei constructed inequalities involving three quantities. 
\end{abstract}

\maketitle

%\tableofcontents

%-------------------------------------------------------------------------

\section{Introduction}
In this paper, we discuss some locally constrained curvature flows and their applications to proving inequalities involving weighted curvature integrals and modified quermassintegrals for smooth h-convex domains in hyperbolic space $\mathbb{H}^{n+1}$.

We consider hyperbolic space as a warped product space $\mathbb{H}^{n+1}=[0,\infty)\times\mathbb{S}^n$ equipped with metric $\overline{g}=\mathrm{d}r^2 + \phi^2(r) g_{\mathbb{S}^n}$, where $\phi(r)=\sinh r,\ r\in[0,\infty)$. Let $\Phi(r)=\cosh r-1$. It's well known that $V=\overline{\nabla}\Phi=\phi(r)\partial_r$ is a conformal killing field.

A smooth bounded domain $\Omega$ in $\mathbb{H}^{n+1}$ is said to be h-convex (resp. strictly h-convex) if the principal curvatures of the boundary $\partial\Omega$ satisfy $\kappa_i\geqslant 1$ (resp. $\kappa_i>1$) for all $i=1,2,\cdots,n$. Let $M$ be a smooth closed hypersurface, denote $\nu$ as its unit outward normal vector, and its support function is defined by $u=\langle V,\nu \rangle$. We say that $M$ is star-shaped if its support function $u$ is positive everywhere on $M$.

\subsection{Curvature flows}
Locally constrained curvature flows have been studied widely in recent years. Brendle, Guan and Li \cite{B-G-L} (see also \cite{Li2021IsoperimetricTI}) introduced the flow $X:M^n\times [0,T)\to\mathbb{N}^{n+1}(K)$
\begin{equation}
	\partial_t X(x,t) = \left(\dfrac{\phi'(r) E_{k-1}(\kappa)}{E_k(\kappa)} - u \right)\nu(x,t),\ \  k=1,2,\cdots,n  \label{equ-bglsflow}
\end{equation}
for a family of smooth closed hypersurfaces in $n+1$ dimensional space forms with constant sectional curvature $K=1,-1$ and discuss some convergence results under several assumptions when $n\geqslant 2$. Here $E_k(\kappa)$ denotes the $k$th normalized elementary symmetric function of principal curvatures $\kappa$ of $M_t=X(M,t)$. 

For $n\geqslant 2$, Hu, Li and Wei \cite{Locallyconstrained} used the tensor maximum principle to verify that the h-convexity preserves along the flow (\ref{equ-bglsflow}) and thus proved the long time existence and exponential convergence to a geodesic sphere if the initial hypersurface is h-convex in hyperbolic space $\mathbb{H}^{n+1}$. 

Kwong, Wei, Wheeler and Wheeler \cite{twodim2022} considered the curve case of (\ref{equ-bglsflow}) in space forms and proved that if the initial curve is strictly convex (i.e. the curvature of curve $\kappa>0$), then the solution curve exists for all time and preserves strict convexity. Moreover, the evolving curve smoothly and exponentially converges to a geodesic circle centered at the origin.

In \cite{Locallyconstrained}, Hu, Li and  Wei introduced a new locally constrained curvature flow $X: M^n\times [0,T) \to \mathbb{H}^{n+1}$,
\begin{equation}
	\partial_t X(x,t) = \left( (\phi'(r)-u)\dfrac{E_{k-1}(\tilde{\kappa})}{E_k(\tilde{\kappa})} -u \right) \nu(x,t),\ \ k=1,2,\cdots,n, \label{equ-flowhy}
\end{equation}
where $\tilde{\kappa}_i=\kappa_i-1$ are shifted principal curvatures of $M_t=X(M,t)$, $E_k(\tilde{\kappa})$ is the normalized $k$th elementary symmetric function of $\tilde{\kappa}$. They proved the following long time existence and convergence results of the flow (\ref{equ-flowhy}) for strictly h-convex hypersurfaces in higher dimensions ($n\geqslant 2$).
\begin{theorem}[{\cite[Theorem 1.6]{Locallyconstrained}}]\label{thm-higherdimension}
	Let $X_0:M^n\to\mathbb{H}^{n+1}\ (n\geqslant 2)$ be a smooth embedding such that $M_0=X_0(M)$ is a smooth, strictly h-convex hypersurface in $\mathbb{H}^{n+1}$. Then the flow (\ref{equ-flowhy}) has a smooth solution for all time $t\in[0,\infty)$, and $M_t=X_t(M)$ is strictly h-convex for each $t>0$ and converges smoothly and exponentially to a geodesic sphere centered at the origin.
\end{theorem}
The convergence results for $n=1$ has not been obtained yet. In the first part of this paper, we aim to make up this gap. When $n=1$, the flow (\ref{equ-flowhy}) becomes
\begin{equation}
	\partial_t X(x,t)= \left( \dfrac{\phi'(r)-u}{\kappa-1}-u \right)\nu(x,t). \label{equ-flow}
\end{equation}
Let's consider a family of immersed closed strictly h-convex curves $X:\mathbb{S}^1\times [0,T)\to \mathbb{H}^2$ satisfying the above flow equation and we will prove the long time existence and convergence results for the flow (\ref{equ-flow}).
\begin{theorem}
	\label{thm-existenceandconvergence}
	Let $X_0:\mathbb{S}^1 \to \mathbb{H}^2$ be a smooth strictly h-convex closed curve containing the origin. There exists a unique solution $X:\mathbb{S}^1\times [0,\infty) \to \mathbb{H}^2$ to the flow \eqref{equ-flow} such that $X(\cdot,0)=X_0$. For each $t\in[0,\infty)$, $\gamma_t=X(\mathbb{S}^1,t)$ is strictly h-convex and converges smoothly and exponentially to a geodesic circle $\gamma_{\infty}=X_{\infty}(\mathbb{S}^1)$ centered at the origin.
\end{theorem}

We remark that for a strictly h-convex initial curve, the short time existence of the flow \eqref{equ-flow} is guaranteed by standard theory of parabolic equations. The long time existence follows by the standard procedure after a priori estimates. To show the convergence, the key ingredient is the Heintze-Karcher type inequality for h-convex curves proved in \cite[Proposition 8.2]{Li2022Hyperbolic}, from which we know that the area of domains enclosed by h-convex curves is monotonically non-decreasing along the flow (\ref{equ-flow}). It also helps us to find the limit curve be a geodesic circle and to prove smooth convergence. In the proof of exponential convergence, inspired by \cite[Section 4.3]{twodim2022} and \cite[Section 9]{Wei2022AFN}, we consider the linearization of the flow \eqref{equ-scalaflow} at infinity and used the spherical harmonics to expand functions in $L^2(\mathbb{S}^1)$ to obtain asymptotic behavior around the infinite time.

\subsection{Applications to geometric inequalities}
In the second part of this paper, we will use the flow (\ref{equ-flow}) to provide a new proof of some geometric inequalities involving weighted curvature integrals and modified quermassintegrals for h-convex curves in $\mathbb{H}^2$. (See Section \ref{subsec-quermassintegrals} for precise definitions.)

%介绍均质积分与平移均质积分不等式已有的一些结果
Many scholars studied the proof of Alexandrov-Fenchel type inequalities in $\mathbb{H}^{n+1}$ by using curvature flow methods. In 2014, Wang and Xia \cite{Wang-Xia2014} proved the quermassintegral inequalities for a h-convex domain $\Omega$ in $\mathbb{H}^{n+1}\ (n\geqslant 2)$,
\begin{equation}
	W_k(\Omega) \geqslant f_k\circ f_{\ell}^{-1}(W_{\ell}(\Omega)),\ \ 0\leqslant \ell<k\leqslant n \label{equ-quermass}
\end{equation}
with equality holds if and only if $\Omega$ is a geodesic ball. Here $f_k:[0,\infty)\to\mathbb{R}^+$ is a monotone function defined by $f_k(r)=W_k(\overline{B}_r)$, the $k$th quermassintegral for the geodesic ball of radius $r$, and $f_{\ell}^{-1}$ is the inverse function of $f_{\ell}$. In 2021, Andrews, Chen and Wei \cite{Ben2021} first introduced modified quermassintegrals for h-convex domains and proved inequalities for a bounded, smooth strictly h-convex domain $\Omega\subset\mathbb{H}^{n+1}\ (n\geqslant 2)$,
\begin{equation}
	\tilde{W}_k(\Omega) \geqslant \tilde{f}_k\circ \tilde{f}_{\ell}^{-1}(\tilde{W}_{\ell}(\Omega)),\ \ 0\leqslant \ell<k\leqslant n. \label{equ-modifiedquer}
\end{equation}
Equality holds in \eqref{equ-modifiedquer} if and only if $\Omega$ is a geodesic ball. Here $\tilde{f}_k:[0,\infty)\to\mathbb{R}^+$ is a monotone function defined by $\tilde{f}_k(r)=\tilde{W}_k(\overline{B}_r)$, the $k$th modified quermassintegral for the geodesic ball of radius $r$, and $\tilde{f}_{\ell}^{-1}$ is the inverse function of $\tilde{f}_{\ell}$.

Later, Hu-Li-Wei \cite{Locallyconstrained} applied the flow (\ref{equ-bglsflow}) and (\ref{equ-flowhy}) for h-convex hypersurfaces in $\mathbb{H}^{n+1}$ to provide a new proof of inequalities (\ref{equ-quermass}) and (\ref{equ-modifiedquer}) respectively. They also used the flow (\ref{equ-flowhy}) to derive inequalities involving weighted curvature integrals and modified quermassintegrals for a strictly h-convex domain $\Omega$ in $\mathbb{H}^{n+1}\ (n\geqslant 2)$ (see \cite[Theorem 1.8]{Locallyconstrained}): For each $k=1,2,\cdots,n$, there holds
\begin{equation}
	\int_{\partial \Omega} \left(\phi'(r)-u \right) E_k(\tilde{\kappa}) \mathrm{d}\mu \geqslant \tilde{g}_k\circ\tilde{f}_k^{-1}(\tilde{W}_k(\Omega)), \label{equ-inequality11}
\end{equation}
where $\tilde{g}_k(r)=\int_{\partial\overline{B}_r}\left(\phi'(r)-u \right) E_k(\tilde{\kappa}) \mathrm{d}\mu$, $\tilde{f}_k(r)=\tilde{W}_k(\overline{B}_{r})$. Equality holds in \eqref{equ-inequality11} if and only if $\Omega$ is a geodesic ball centered at the origin.

Recently, Kwong and Wei \cite{KWONG2023109213} investigated a family of three-term geometric inequalities involving weighted curvature integrals and quermassintegrals for hypersurfaces in space forms. As a special case of \cite[Theorem 5.4]{KWONG2023109213}, they proved that for a smooth bounded h-convex domain $\Omega$ in $\mathbb{H}^{n+1}\ (n\geqslant 2)$, there holds
\begin{equation}
	\int_{\partial\Omega} \Phi E_k(\kappa) \mathrm{d}\mu + kW_{k-1}(\Omega) \geqslant (\xi_k+kf_{k-1})\circ f_{\ell}^{-1} (W_{\ell}(\Omega))   \label{equ-threeterms}
\end{equation}
for all $k=1,2,\cdots,n$ and $l=0,1,\cdots,k$. Equality holds in \eqref{equ-threeterms} if and only if $\Omega$ is a geodesic ball centered at the origin. $\xi_k(r)=\int_{\partial\overline{B}_r} \Phi E_k(\kappa) \mathrm{d}\mu$, $f_k(r)=W_k(\overline{B}_r)$.

Li and Xu proposed a conjecture of horospherical $p$-Minkowski inequalities in \cite[Conjecture 10.5]{Li2022Hyperbolic}. Especially, they proved some inequalities which are special cases of the conjecture from a horospherically convex geometric point of view. They used the flow \eqref{equ-flowhy} and the convergence result (Theorem \ref{thm-higherdimension}) to prove the following three-term inequality involving weighted curvature integrals and modified quermassintegrals for strictly h-convex hypersurfaces. We state their result in our setting.

\begin{theorem}[{\cite[Theorem 10.9]{Li2022Hyperbolic}}]
	\label{thm-inequalityhighdim}
	Let $M$ be a smooth strictly h-convex hypersurface enclosing a domain $\Omega$ in $\mathbb{H}^{n+1}\ (n\geqslant 2)$. For any $k=1,2,\cdots,n$, there holds
	\begin{equation}
		\int_M (\Phi-u) E_k(\tilde{\kappa}) \mathrm{d}\mu + (n-2k) \tilde{W}_k(\Omega) \geqslant \tilde{h}_k\circ \tilde{f}_k^{-1}(\tilde{W}_k(\Omega)), \label{equ-inequality1}
	\end{equation}
	where $\tilde{h}_k(r)= \int_{\partial\overline{B}_r}(\Phi-u)E_k(\tilde{\kappa}) \mathrm{d}\mu + (n-2k) \tilde{W}_k(\overline{B}_r)$, $\tilde{f}_k(r)=\tilde{W}_k(\overline{B}_r)$, $\overline{B}_r$ is a geodesic ball centered at the origin, $\tilde{f}_k^{-1}$ is the inverse function of $\tilde{f}_k(r)$. Equality holds if and only if $M$ is a geodesic sphere centered at the origin.
\end{theorem}

Note that when $n=1$, we have $\tilde{W}_1(\Omega)=L[\gamma]-A[\gamma]$ by \eqref{equ-defofmodiquer}, where $L[\gamma]$ denotes the length of the curve $\gamma$, $A[\gamma]$ denotes the area of the domain $\Omega$ enclosed by $\gamma$. It has been proved in \cite[Theorem 10.13]{Li2022Hyperbolic} that \eqref{equ-inequality1} holds for a smooth strictly h-convex curve in $\mathbb{H}^2$ by using horospherical Gauss map.

Motivated by the work \cite{KWONG2023109213,Li2022Hyperbolic}, we provide a new proof of inequality \eqref{equ-inequality1} for h-convex curves in $\mathbb{H}^{2}$ by using the flow (\ref{equ-flow}) and the convergence result (Theorem \ref{thm-existenceandconvergence}).

\begin{theorem}
	\label{thm-inequality2dim}
	Let $\gamma$ be a smooth h-convex closed curve enclosing a domain $\Omega$ in $\mathbb{H}^2$. There holds
	\begin{equation}
		\int_{\gamma} (\Phi-u)(\kappa-1) \mathrm{d}s + (L[\gamma]-A[\gamma]) \geqslant \dfrac{1}{2\pi}(L[\gamma]-A[\gamma])^2 \label{equ-ineq}
	\end{equation}
	with equality if and only if $\gamma$ is a geodesic circle centered at the origin. Here $\Phi(r)=\cosh r - 1$, $u=\langle V,\nu \rangle$ is the support function of $\gamma$, $L[\gamma]$ denotes the length of $\gamma$, $A[\gamma]$ denotes the area of the domain $\Omega$ enclosed by $\gamma$.
\end{theorem}

We first use the flow \eqref{equ-flow} to prove inequality \eqref{equ-ineq} for strictly h-convex curves in $\mathbb{H}^2$. Then inequality \eqref{equ-ineq} for h-convex curves follows by an approximation argument following \cite{Guan2009TheQI} and \cite{twodim2022}.

%For $n\geqslant 2$, we have the following:

%From Remark \ref{rmk}, we know that (\ref{equ-inequality1}) is sharper than (\ref{equ-inequality11}).

\subsection{Stability results}
We finally consider the stability of inequalities (\ref{equ-ineq}) for curves, (\ref{equ-modifiedquer}) and (\ref{equ-inequality1}) for hypersurfaces characterized by the Hausdorff distance between the hypersurface and an appropriate geodesic sphere. Recall that
\begin{equation}
	\mathrm{dist}(K,L)=\inf\{\delta>0:K\subset B_{\delta}(L)\ \mbox{and}\ L\subset B_{\delta}(K)\}
\end{equation}
is the Hausdorff distance between two compact sets $K$ and $L$.

For $n=1$, we will use Theorem \ref{thm-existenceandconvergence} to prove the stability of inequality (\ref{equ-ineq}) for strictly h-convex curves.
\begin{theorem}
	\label{thm-stabilitycurve}
	Let $\gamma$ be a smooth strictly h-convex closed curve in $\mathbb{H}^2$ enclosing a domain $\Omega$. Regard $\gamma=\{(\rho(\theta),\theta):\theta\in\mathbb{S}^1\}$ as a graph on $\mathbb{S}^1$. There exists a constant $C(\rho,\max\limits_{\gamma}|\rho_{\theta}|,\min\limits_{\gamma}\kappa)$ and a geodesic circle $\gamma_{\mathbb{H}}$, such that
	\begin{equation}
		\mathrm{dist}(\gamma,\gamma_{\mathbb{H}}) \leqslant C f\left[ \int_{\gamma} (\Phi-u)(\kappa-1) \mathrm{d}s + (L[\gamma]-A[\gamma]) - \dfrac{1}{2\pi}(L[\gamma]-A[\gamma])^2 \right],  \label{equ-curvestability}
	\end{equation}
	where $f(s)=s^{\frac{1}{2}}+s^{\frac{1}{6}}\ (s\geqslant 0)$, and $\mathrm{dist}(\gamma,\gamma_{\mathbb{H}})$ denotes the Hausdorff distance between $\gamma$ and $\gamma_{\mathbb{H}}$.
\end{theorem}

The proof relies on the Poincar\'e inequality, which is used to characterize the distance between the curve and a geodesic circle.

In \cite{Sahjwani2023StabilityOT}, Sahjwani and Scheuer used flow (\ref{equ-bglsflow}) to prove the stability of quermassintegral inequalities (\ref{equ-quermass}) when $l=k-1$ for h-convex domains in $\mathbb{H}^{n+1}\ (n\geqslant 2)$. Their proof relied on the following nearly umbilical theorem.

\begin{theorem}[{\cite[Theorem 1.2]{DeRosa2018QuantitativeSF}}]
	\label{thm-stability}
	Let $n\geqslant 2$, $p\in(1,\infty)$. There exists $\varepsilon(n,p)>0$, if $\Sigma$ is a closed convex hypersurface in $\mathbb{R}^{n+1}$ satisfying
	\begin{equation}
		|\Sigma|=|\mathbb{S}^n|,\ \ \int_{\Sigma}|\mathring{A}|^p \mathrm{d}\mu\leqslant \varepsilon,
	\end{equation}
	then there exist a smooth parametrization $\varphi:\mathbb{S}^n\to\Sigma\subset\mathbb{R}^{n+1}$ and a point $\mathcal{P}\in\mathbb{R}^{n+1}$ such that the following estimate holds:
	\begin{equation}
		\|\varphi-\mathrm{id}-\mathcal{P}\|_{W^{2,p}(\mathbb{S}^n)} \leqslant C(n,p) \|\mathring{A}\|_{L^p(\Sigma)},
	\end{equation}
	where $\mathrm{id}:\mathbb{S}^n\to\mathbb{S}^n$ is the identity map, $\mathring{A}$ is the trace free Weingarten matrix of $\Sigma$.
\end{theorem}

For $n\geqslant 2$, inspired by the work \cite{Sahjwani2023StabilityOT}, we will use the flow (\ref{equ-flowhy}) with the convergence result (Theorem \ref{thm-higherdimension}) and the nearly umbilical result (Theorem \ref{thm-stability}) to discuss the stability of inequalities (\ref{equ-modifiedquer}) and (\ref{equ-inequality1}) for strictly h-convex hypersurfaces.

\begin{theorem}
	\label{thm-stabilityhypersurface}
	Let $M$ be a smooth strictly h-convex closed hypersurface in $\mathbb{H}^{n+1}\ (n\geqslant 2)$ enclosing a domain $\Omega$. For any $1\leqslant k\leqslant n-1$, there exist a constant $C=C(\rho_{-}(\Omega),\mathrm{dist}(\tilde{\kappa},\partial\Gamma_k^+),M,n,k)$ and a geodesic sphere $S_{\mathbb{H}}$, such that
	\begin{equation}
		\mathrm{dist}(\partial\Omega,S_{\mathbb{H}}) \leqslant C f\left[ \tilde{W}_{k+1}(\Omega) - \tilde{f}_{k+1}\circ \tilde{f}_{k}^{-1}\left( \tilde{W}_{k}(\Omega) \right) \right] \label{equ-stabi2}
	\end{equation}
	and
	\begin{equation}
		\mathrm{dist}(\partial\Omega,S_{\mathbb{H}}) \leqslant C f\left[ \int_{M} (\Phi-u) E_k(\tilde{\kappa}) \mathrm{d}\mu + (n-2k)\tilde{W}_k(\Omega) - \tilde{h}_k\circ \tilde{f}_k^{-1}\left( \tilde{W}_k(\Omega) \right) \right] \label{equ-stabi}
	\end{equation}
	hold, where $f(s)=s^{\frac{1}{2}}+s^{\frac{1}{4}}\ (s\geqslant 0)$, $\mathrm{dist}(\partial\Omega,S_{\mathbb{H}})$ is the Hausdorff distance between $\partial\Omega$ and $S_{\mathbb{H}}$.
\end{theorem}

The paper is organized as follows. In Section \ref{section-2}, we will first gather some useful properties of elementary symmetric functions, quermassintegrals and modified quermassintegrals. Then we summarize important formulas for hypersurfaces in hyperbolic space and general evolution equations of curvature flows. In Section \ref{section-3}, we will prove Theorem \ref{thm-existenceandconvergence}, the long time existence and convergence results for the curve flow (\ref{equ-flow}). In Section \ref{section-4}, we will prove geometric inequalities Theorem \ref{thm-inequality2dim} and include the complete proof of Theorem \ref{thm-inequalityhighdim}. In Section \ref{section-5}, we will prove stability results Theorems \ref{thm-stabilitycurve} and \ref{thm-stabilityhypersurface}.

\begin{ack}
	The authors would like to thank Professor Yong Wei for his helpful discussions and support. The authors also appreciate Dr. Botong Xu for his advice on improving this work. This work is supported by National Key Research and Development Program of China 2021YFA1001800.
\end{ack}

\section{Preliminaries}
\label{section-2}
In this section, we will make some preparations for our proof, including properties of elementary symmetric functions, definitions for quermassintegrals and modified quermassintegrals, some formulas for hypersurfaces in hyperbolic space and general evolution equations for curvature flows.

\subsection{Elementary symmetric functions}
Here we review some properties of elementary symmetric functions. For integer $k=1,2,\cdots,n$, for any point $\kappa=(\kappa_1,\kappa_2,\cdots,\kappa_n)\in\mathbb{R}^n$, the normalized $k$th elementary symmetric polynomial $E_k(\kappa)$ is defined by
\begin{equation}
	E_k(\kappa) = {n\choose k}^{-1}\sum\limits_{1\leqslant i_1<i_2<\cdots<i_k\leqslant n} \kappa_{i_1}\kappa_{i_2}\cdots\kappa_{i_{k}}
\end{equation}
and has the convention that $E_0=1$ and $E_k(\kappa)=0$ for $k>n$. The definition can be extended to symmetric matrices. Let $A=(a_{ij})$ be a $n\times n$ symmetric matrix, denote $\kappa=\kappa(A)$ as the eigenvalue vector of $A$. Define $E_k(A)=E_k(\kappa(A))$, and we have
\begin{equation}
	E_k(A)=\dfrac{(n-m)!}{n!}\delta_{i_1i_2\cdots i_k}^{j_1j_2\cdots j_k}a_{i_1j_1}a_{i_2j_2}\cdots a_{i_kj_k},\ \ k=1,2,\cdots,n.
\end{equation}

We include some useful properties here but omit their proofs. For more details, we may refer to \cite[Section 2]{Guan2014}.
\begin{lemma}\label{lem-elementarysymmetric}
	Denote $\dot{E}^{ij}_k=\dfrac{\partial E_k(A)}{\partial a_{ij}}$, then we have the following identities.
	\begin{eqnarray}
		&&\sum\limits_{i,j}\dot{E}^{ij}_k\delta^j_i = k E_{k-1}(A),\notag\\
		&&\sum\limits_{i,j}\dot{E}^{ij}_k a_{ij} = k E_{k}(A),\notag\\
		&&\sum\limits_{i,j}\dot{E}^{ij}_k(A^2)_{ij} = n E_1(A)E_k(A) - (n-k)E_{k+1}(A).
	\end{eqnarray}
\end{lemma}
For each $k=1,2,\cdots,n$, let
$$\Gamma_k^{+}=\{x\in\mathbb{R}^n:E_i(x)>0,\ i=1,2,\cdots,k\}$$
be the Garding cone.
\begin{lemma}[Newton-MacLaurin inequality]
	If $\kappa\in\Gamma_k^+$, we have the following inequality
	\begin{equation}
		E_{k+1}(\kappa)E_{l-1}(\kappa) \leqslant E_l(\kappa)E_k(\kappa),\ \ 1\leqslant l\leqslant k.  \label{equ-MewtonMacL}
	\end{equation}
    Equality holds if and only if $\kappa_1=\kappa_2=\cdots=\kappa_n$.
\end{lemma}

The following divergence free property will be used frequently in integration by parts.
\begin{lemma}
	\label{lem-div}
	If $A=(a_{ij})$ is a Codazzi tensor on Riemannian manifold $(M^n,g)$, let $\{e_1,e_2,\cdots,e_n\}$ be a local frame, then
	\begin{equation}
		\sum\limits_{j}\nabla_j\dot{E}_k^{ij}(A) = 0. \label{equ-divergencefree}
	\end{equation}
\end{lemma}

\subsection{Quermassintegrals and modified quermassintegrals}
\label{subsec-quermassintegrals}
Recall that for a convex domain $\Omega$ in hyperbolic space $\mathbb{H}^{n+1}$, the quermassintegrals of $\Omega$ are defined as follows (see \cite[Definition 2.1]{G-Solanes05}):
\begin{equation}
	W_k(\Omega) = \dfrac{\omega_{k-1}\cdots\omega_0}{\omega_{n-1}\cdots\omega_{n-k}}\int_{\mathcal{L}_k} \chi(L_k\cap \Omega) \mathrm{d}L_k,\ k=1,2,\cdots,n,
\end{equation}
where $\omega_k=|\mathbb{S}^k|$ denotes the area of $k$-dimensional sphere, $\mathcal{L}_k$ is the space of $k$-dimensional totally geodesic subspaces $L_k$ in $\mathbb{H}^{n+1}$. The function $\chi$ is defined to be 1 if $L_k\cap \Omega\neq\varnothing$ and to be 0 otherwise. Furthermore, we set
\begin{equation}
	W_0(\Omega) = |\Omega|,\ \ W_{n+1}(\Omega) = |\mathbb{B}^{n+1}| = \dfrac{\omega_n}{n+1}.
\end{equation}
If the boundary of $\Omega$ is smooth, the quermassintegrals and the curvature integrals of a smooth domain $\Omega$ in $\mathbb{H}^{n+1}$ are related by the following equations:
\begin{eqnarray}
	&&W_1(\Omega)=\dfrac{|\partial\Omega|}{n},\notag\\
	&&\int_{\partial\Omega} E_k(\kappa) \mathrm{d}\mu = (n-k) W_{k+1}(\Omega) + k W_{k-1}(\Omega),\ k=1,2,\cdots,n-1,\\
	&&\int_{\partial\Omega} E_n(\kappa) \mathrm{d}\mu = \omega_n + n W_{n-1}(\Omega).
\end{eqnarray}

Recently, Andrews-Chen-Wei \cite{Ben2021} introduced modified quermassintegrals for h-convex domains in $\mathbb{H}^{n+1}$:
\begin{equation}
	\label{equ-defofmodiquer}
	\tilde{W}_k(\Omega):=\sum\limits_{i=0}^k (-1)^{k-i}{k\choose i}W_i(\Omega),\ \  k=0,1,\cdots,n.
\end{equation}
We know that $\tilde{W}_k(\Omega)$ is a linear combination of the quermassintegrals of $\Omega$. It can be verified that (see \cite[Lemma 10.1]{Li2022Hyperbolic})
\begin{eqnarray}
	\int_{\partial\Omega} E_k(\tilde{\kappa}) \mathrm{d}\mu &=& (n-k) \tilde{W}_{k+1}(\Omega) + (n-2k) \tilde{W}_k(\Omega),\ k=1,2,\cdots, n-1,   \label{equ-modifiedrelation}\\
	\int_{\partial\Omega}E_n(\tilde{\kappa})\mathrm{d}\mu&=&\omega_n-n\tilde{W}_n(\Omega),\label{equ-modifiedrelation2}
\end{eqnarray}
where $\tilde{\kappa}_i=\kappa_i-1$ are the shifted principal curvatures of the hypersurface $\partial\Omega$. These new quermassintegrals satisfy the following nice variational formula (see \cite[Lemma 2.4]{Ben2021})
\begin{equation}
	\dfrac{\partial}{\partial t}\tilde{W}_k(\Omega_t) = \int_{\partial\Omega_t} \eta E_{k}(\tilde{\kappa}) \mathrm{d}\mu_t  \label{equ-ptwk}
\end{equation}
along any variation in the direction of outward normal with speed function $\eta$.

\subsection{Hypersurfaces in hyperbolic space}
In this paper, we consider $\mathbb{H}^{n+1}=\mathbb{S}^n\times [0,\infty)$ as a warped product manifold equipped with metric
\begin{equation}
	\overline{g}=\mathrm{d}r^2+\phi^2(r)g_{\mathbb{S}^n}, 
\end{equation}
where $\phi(r)=\sinh r$, $r\in[0,\infty)$, $g_{\mathbb{S}^n}$ denotes the standard metric on $\mathbb{S}^n$. Let
\begin{equation}
	\Phi(r)=\int_0^r\phi(s)\mathrm{d}s=\cosh r-1.
\end{equation}
It's well known that $V=\overline{\nabla}\Phi=\phi(r)\partial_r$ is a conformal Killing field on $\mathbb{H}^{n+1}$, that is
\begin{equation}
	\overline{\nabla} V = \phi'(r)\overline{g}.
\end{equation}

Let $n=1$. Let $\gamma=X(\theta)$ be a smooth closed curve with its position vector $X:\mathbb{S}^1\to\mathbb{H}^2$. Denote $v=|X_{\theta}|=\sqrt{\phi^2(r)+r_{\theta}^2}$, and define the arc-length parameter $s$ by
$$\dfrac{\partial}{\partial s}=\dfrac{1}{v}\dfrac{\partial}{\partial\theta},$$
then $T=\dfrac{\partial X}{\partial s}=\dfrac{1}{v}X_{\theta}$ is the unit tangent vector. Denote $\nu$ as the unit outward normal vector and $\kappa$ as the curvature of $\gamma$, we have the Frenet formula
\begin{equation}
	\overline{\nabla}_T T=-\kappa\nu,\ \ \overline{\nabla}_T \nu=\kappa T.
\end{equation}

It has been verified in \cite[Proposition 3.2]{twodim2022} that the following formulas hold for smooth curves, which can be regarded as the special case for higher dimensions.
\begin{lemma} Let $\gamma$ be a smooth closed curve in $\mathbb{H}^2$ with $s$ be its arc-length parameter. Let $u=\langle V,\nu\rangle$ be its support function, $\kappa$ be its curvature, then
	\begin{eqnarray}
		&&\dfrac{\partial \phi'}{\partial s} =\langle V,T \rangle,\ \ \dfrac{\partial^2 \phi'}{\partial s^2}= \phi'-\kappa u,\label{equ-Phi'}\\
		&&\dfrac{\partial u}{\partial s} = \kappa \langle V,T \rangle,\ \,\dfrac{\partial^2 u}{\partial s^2} = \dfrac{\partial\kappa}{\partial s}\langle V,T \rangle + \kappa (\phi'-\kappa u).
	\end{eqnarray}
    Moreover, 
    \begin{equation}
    	\int_{\gamma}\left(\phi'(r)-\kappa u \right)\mathrm{d}s = 0.  \label{equ-curvemink}
    \end{equation}
\end{lemma}

Let $n\geqslant 2$. Let $M$ be a smooth closed hypersurface in $\mathbb{H}^{n+1}$ with induced metric $g$ and unit outward normal vector $\nu$. Let $\{x^1,x^2,\cdots,x^n\}$ be a local coordinate system on $M$, we denote $g_{ij}=g(\partial_i,\partial_j)$ and its second fundamental form $h_{ij}=\langle \overline{\nabla}_{\partial_i}\nu,\partial_j \rangle$. The Weingarten matrix is denoted by $A=(h^i_j)$, where $h^i_j=g^{ik}h_{kj}$, and $(g^{ij})$ is the inverse matrix of $(g_{ij})$. The principal curvatures $\kappa=(\kappa_1,\kappa_2,\cdots,\kappa_n)$ are eigenvalues of the Weingarten matrix $A$. The connection on $M$ is denoted by $\nabla$.

The following formulas hold for smooth hypersurfaces in $\mathbb{H}^{n+1}$, see \cite[Lemmas 2.2 and 2.6]{Guan2013AMC}.
\begin{lemma}\label{lem-hypersurface}
	Let $(M,g)$ be a smooth hypersurface in $\mathbb{H}^{n+1}$ with local frame $\{\partial_1,\cdots,\partial_n\}$, then $\Phi|_M$ satisfies
	\begin{equation}
		\nabla_i\Phi =\langle V,e_i \rangle,\ \ \nabla_j\nabla_i \Phi= \phi' g_{ij} - u h_{ij},
	\end{equation}
    and the support function $u=\langle V,\nu\rangle$ satisfies
    \begin{equation}
    	\nabla_i u = \langle V,e_l \rangle h_i^l,\ \ \nabla_j\nabla_i u= \langle V,\nabla h_{ij} \rangle + \phi' h_{ij} - u(h^2)_{ij},
    \end{equation}
    where $(h^2)_{ij}=h_{ik}h^k_j$.
\end{lemma}

We also have the following Minkowski formulas.
\begin{lemma}[{\cite[Proposition 2.5]{Guan2013AMC}}]
		Let $M$ be a smooth closed hypersurface in $\mathbb{H}^{n+1}$. Then
		\begin{equation}
			\int_M \phi' E_k(\kappa) \mathrm{d}\mu = \int_M u E_{k+1}(\kappa) \mathrm{d}\mu,\ \ k=0,1,\cdots,n-1,
		\end{equation} 
		where $E_k(\kappa)$ is the $k$th mean curvature of $M$.
\end{lemma}
	For shifted $k$th mean curvature, the Minkowski type formulas also hold.
	\begin{lemma}[{\cite[Lemma 2.6]{Locallyconstrained}}]
		Let $M$ be a smooth closed hypersurface in $\mathbb{H}^{n+1}$. Let $\tilde{\kappa}_i=\kappa_i-1$ be the shifted principal curvatures of $M$, then
		\begin{equation}
			\int_M (\phi'-u)E_k(\tilde{\kappa}) \mathrm{d}\mu = \int_M uE_{k+1}(\tilde{\kappa}) \mathrm{d}\mu,\ \ k=0,1,\cdots,n-1.  \label{equ-shiftedminkowski}
		\end{equation}
	\end{lemma}

\subsection{General evolution equations}
Let $X:M^n\times [0,T)\to \mathbb{H}^{n+1}$  be a family of immersed hypersurfaces satisfying the outward normal variation with speed $\eta$
\begin{equation}
	\partial_t X(x,t) = \eta(x,t) \nu (x,t). \label{equ-etaflow}
\end{equation}
Denote $M_t=X_t(M)$ and let $\{\partial_1,\partial_2,\cdots,\partial_n\}$ be local frame on $M_t$. Then the induced metric $g_{ij}$, the unit outward normal vector $\nu$, the second fundamental form $h_{ij}$, the weingarten tensor $h^i_j$, the induced area element $\mathrm{d}\mu_t$ evolve as follows (see \cite[Lemma 3.1]{Guan2013AMC}):
\begin{eqnarray}
	\dfrac{\partial}{\partial t} g_{ij} &=& 2\eta h_{ij},\\
	\dfrac{\partial}{\partial t} \nu &=& -\nabla\eta,\\
	\dfrac{\partial}{\partial t} h_{ij} &=& -\nabla_j\nabla_i \eta + \eta(h^2)_{ij}+\eta g_{ij},\\
	\dfrac{\partial}{\partial t} h^i_j &=& -\nabla_j\nabla^i \eta - \eta (h^2)^i_j + \eta\delta^i_j,\\
	\dfrac{\partial}{\partial t} \mathrm{d}\mu_t &=& \eta H \mathrm{d}\mu_t, \label{equ-areaevolution}
\end{eqnarray}
where $H=nE_1(\kappa)$ is the mean curvature of $M_t$, $\nabla$ denotes the Levi-Civita connection with respect to the induced metric $g_{ij}$ on $M_t$.

Moreover, along the flow (\ref{equ-etaflow}), $\Phi(r)=\cosh r-1$ and the support function $u$ evolve as follows:
\begin{eqnarray}
	\dfrac{\partial}{\partial t} \Phi &=& \left\langle \overline{\nabla} \Phi, \partial_t X \right\rangle = \langle V,\eta\nu \rangle=\eta u,\label{equ-ptPhi}\\
	\dfrac{\partial}{\partial t} u &=&  \langle \overline{\nabla}_{\partial_t X} V,\nu \rangle + \langle V,\partial_t \nu \rangle = \eta\phi'(r)-\langle V,\nabla \eta \rangle. \label{equ-ptu}
\end{eqnarray}

Let $\Omega_t$ be the domain enclosed by $M_t$, then the volume of $\Omega_t$ denoted by $\mathrm{Vol}(\Omega_t)$ evolves by
\begin{equation}
	\dfrac{\partial}{\partial t}\mathrm{Vol}(\Omega_t) = \int_{M_t} \eta \mathrm{d}\mu_t. \label{equ-volumeevolution}
\end{equation}

Now we would like to derive the evolution equation for the shifted $k$th mean curvature $E_k(\tilde{\kappa})$ (see also \cite[Section 9.1]{Locallyconstrained}). Set $S^i_j=h^i_j-\delta^i_j$, then the shifted weingarten tensor satisfies
\begin{eqnarray}
	\dfrac{\partial}{\partial t} S^i_j &=& -\nabla_j\nabla^i\eta - \eta\left( (S^2)^i_j+2S^i_j \right).
\end{eqnarray}
Denote $\dot{E}_k^{ij}(\tilde{\kappa}) = \dfrac{\partial E_k(\tilde{\kappa})}{\partial S^i_j}$, then by Lemma \ref{lem-elementarysymmetric}, we have
\begin{eqnarray}
	\dot{E}_k^{ij}(\tilde{\kappa})\delta^i_j &=& k E_{k-1}(\tilde{\kappa}),\\
	\dot{E}_k^{ij}(\tilde{\kappa})S^i_j &=& kE_k(\tilde{\kappa}),\\
	\dot{E}_k^{ij}(\tilde{\kappa})\left(S^2\right)^i_j &=& n E_1(\tilde{\kappa}) E_k(\tilde{\kappa}) - (n-k) E_{k+1}(\tilde{\kappa}).
\end{eqnarray}
From Lemma \ref{lem-hypersurface}, we get
\begin{eqnarray}
	\nabla_j\nabla^i \Phi	&=& (\phi'-u)\delta^i_j - u S^i_j,\\
	\nabla_j\nabla^i u &=& \langle V,\nabla S^i_j \rangle - u \left( (S^2)^i_j+2S^i_j + \delta^i_j\right) + \phi'(S^i_j+\delta^i_j),\\
	\nabla_j\nabla^i(\Phi-u) &=& -\phi'S^i_j-\langle V,\nabla S^i_j \rangle + u \left( (S^2)^i_j+S^i_j \right).
\end{eqnarray}
Combining the above identities, we obtain
\begin{equation}
	\dot{E}_k^{ij}(\tilde{\kappa})\nabla_j\nabla^i \Phi = k \left( (\phi'-u)E_{k-1}(\tilde{\kappa}) - u E_k(\tilde{\kappa}) \right), \label{equ-ekdotijphi}
\end{equation}
and the shifted $k$th mean curvature evolves as follows:
\begin{eqnarray}
\dfrac{\partial}{\partial t} E_k(\tilde{\kappa}) &=& -\dot{E}_k^{ij}(\tilde{\kappa})\nabla_j\nabla^i\eta - \left( nE_1(\tilde{\kappa})E_k(\tilde{\kappa}) - (n-k)E_{k+1}(\tilde{\kappa}) \right)\eta - 2k E_k(\tilde{\kappa})\eta.\label{equ-ptek}
\end{eqnarray}

\section{Flow for strictly h-convex curves in hyperbolic space}
\label{section-3}
In this section, we will prove Theorem \ref{thm-existenceandconvergence}, the long time existence and convergence results for the flow (\ref{equ-flow}).

Recall that Heintze-Karcher type inequality for h-convex curves is obtained in \cite[Proposition 8.2]{Li2022Hyperbolic}. Let $\gamma$ be a smooth strictly h-convex curve in $\mathbb{H}^2$ enclosing a domain $\Omega$, then
\begin{equation}
	\label{equ-heintze-karcher-type}
	\int_{\partial\Omega} \dfrac{\phi'(r)-u}{\kappa-1} - u \mathrm{d}s \geqslant 0.
\end{equation}
Equality holds in \eqref{equ-heintze-karcher-type} if and only if $\Omega$ is a geodesic ball. This is an important ingredient in our proof.

We will first verify that the difference between the length of curve and the area of the domain enclosed by the curve $L[\gamma_t]-A[\gamma_t]$ preserves along the flow (\ref{equ-flow}).

\begin{lemma}\label{lemma-la}
	Along the flow \eqref{equ-flow}, the smooth closed curve $\gamma_t$ satisfies
	\begin{eqnarray}
		\dfrac{\partial}{\partial t} \left(L[\gamma_t] - A[\gamma_t] \right) = 0, \label{equ-mono1}
	\end{eqnarray}
	where $L[\gamma_t]$ is the length of $\gamma_t$, $A[\gamma_t]$ denotes the area of the domain enclosed by $\gamma_t$.
\end{lemma}
\begin{proof}
	Set $\eta=\dfrac{\phi'(r)-u}{\kappa-1}-u$. By the evolution equation of length \eqref{equ-areaevolution} and area \eqref{equ-volumeevolution}, and Minkowski formula (\ref{equ-curvemink}), we know
	\begin{equation}
		\dfrac{\partial}{\partial t} \left( L[\gamma_t]-A[\gamma_t] \right) =\int_{\gamma_t} \eta(\kappa-1) \mathrm{d}s = \int_{\gamma_t} \left(\phi'(r)-\kappa u \right) \mathrm{d}s = 0,
	\end{equation}
    as desired.
\end{proof}

\subsection{A priori estimates} We derive some a priori estimates for the flow (\ref{equ-flow}).
\begin{proposition}[Curvature Estimate]
	\label{prop-curvature}
	Let $X_0$ be a strictly h-convex curve, then $\gamma_t$ is strictly h-convex along the flow \eqref{equ-flow}. Furthermore, we have the uniform curvature estimate
	\begin{equation}
		\min \kappa_0 \leqslant \kappa(\cdot,t) \leqslant \max \kappa_0, \label{equ-kappa}
	\end{equation}
	where $\kappa_0$ is the curvature of $X_0$.
\end{proposition}
\begin{proof}
	We calculate the evolution equation of curvature $\kappa(\cdot,t)$ for $\gamma_t$. Let $s$ be the arc-length parameter. After a direct computation, we get
	\begin{eqnarray}
		\dfrac{\partial\kappa}{\partial t} &=& \dfrac{1}{\kappa-1}\left(-\dfrac{\partial^2 \phi'}{\partial s^2} + \dfrac{\partial^2 u}{\partial s^2} \right) + \dfrac{2}{(\kappa-1)^2}\left( \dfrac{\partial\phi'}{\partial s} - \dfrac{\partial u}{\partial s} \right) \dfrac{\partial\kappa}{\partial s} + \dfrac{\phi'-u}{(\kappa-1)^2} \dfrac{\partial^2 \kappa}{\partial s^2}\notag\\
		&&-\dfrac{2(\phi'-u)}{(\kappa-1)^3}\left( \dfrac{\partial\kappa}{\partial s} \right)^2 + \dfrac{\partial^2 u}{\partial s^2} - (\kappa^2-1)\left( \dfrac{\phi'-u}{\kappa-1}-u \right).
	\end{eqnarray}
	By using
	\begin{eqnarray}
		\dfrac{\partial\phi'}{\partial s}=\langle V,T \rangle,\ \ \dfrac{\partial^2 \phi'}{\partial s^2} = \phi'-\kappa u, \ \ \dfrac{\partial u}{\partial s} = \kappa \langle V,T \rangle,\ \ \dfrac{\partial^2 u}{\partial s^2} = \dfrac{\partial\kappa}{\partial s}\langle V,T \rangle + \kappa (\phi'-\kappa u),
	\end{eqnarray}
	we get
	\begin{equation}
		\dfrac{\partial\kappa}{\partial t} = \dfrac{\phi'-u}{(\kappa-1)^2}\dfrac{\partial^2\kappa}{\partial s^2} + \dfrac{\partial\kappa}{\partial s}\left[ \dfrac{\partial}{\partial s}\left( \dfrac{\phi'-u}{(\kappa-1)^2} \right) + \langle V,T \rangle\right].
	\end{equation}
	Note that $\phi'-u>0$, by the maximum principle, we have
	\begin{equation}
		\dfrac{\partial}{\partial t} \max\limits_{x\in\mathbb{S}^1} \kappa(x,t)\leqslant 0,\ \dfrac{\partial}{\partial t} \min\limits_{x\in\mathbb{S}^1} \kappa(x,t)\geqslant 0.
	\end{equation}
	Then (\ref{equ-kappa}) follows and $\gamma_t$ remains strictly h-convex.
\end{proof}

Note that h-convexity implies star-shapedness with respect to some point inside the domain. In the following, we will regard the solution curve of the flow (\ref{equ-flow}) as the graph over  $\mathbb{S}^1$, $\gamma_t=\{(\rho(\theta),\theta):\theta\in\mathbb{S}^1\}$, where $\rho:\mathbb{S}^1\to (0,\infty)$ is the radial distance. The unit outward normal vector is
\begin{equation}
	\nu=\dfrac{1}{\sqrt{1+\rho_{\theta}^2\phi(\rho)^{-2}}}\left( \partial\rho-\rho_{\theta}\phi(\rho)^{-2}\partial_\theta \right),
\end{equation}
the curvature is
\begin{equation}
	\kappa = \dfrac{\phi(\rho)^2\phi'(\rho)+2\rho_{\theta}^2\phi'(\rho)-\rho_{\theta\theta}\phi(\rho)}{\left(\phi(\rho)^2+ \rho_{\theta}^2\right)^{\frac{3}{2}}}, \label{equ-ka}
\end{equation}
the support function is
\begin{equation}
	u=\langle V,\nu \rangle = \dfrac{\phi(\rho)}{\sqrt{1+\rho_{\theta}^2\phi(\rho)^{-2}}}.\label{equ-support}
\end{equation}
For a family of curves which can be regarded as the graph over $\mathbb{S}^1$, we can write
\begin{equation*}
	X(x,t)=(\rho(\psi(x,t),t),\psi(x,t)),
\end{equation*}
where $\psi(\cdot,t):\mathbb{S}^1\to\mathbb{S}^1$ is a family of diffeomorphism and $\rho(\cdot,t):\mathbb{S}^1\to(0,\infty)$ is the radial distance. Then
\begin{equation*}
	\dfrac{\partial X}{\partial t} = \left( \dfrac{\partial\rho}{\partial\theta} \dfrac{\partial\psi}{\partial t} + \dfrac{\partial\rho}{\partial t} \right)\dfrac{\partial}{\partial \rho} + \dfrac{\partial\psi}{\partial t}\dfrac{\partial}{\partial\theta}.
\end{equation*}
By a direct computation,
\begin{equation}
	\left\langle \dfrac{\partial X}{\partial t},\nu \right\rangle = \dfrac{1}{\sqrt{1+\rho_{\theta}^2\phi(\rho)^{-2}}}\dfrac{\partial\rho}{\partial t}.
\end{equation}
Therefore, we express the flow (\ref{equ-flow}) equivalently as the initial value problem:
\begin{equation}
    \left\lbrace\begin{array}{ll}
	\dfrac{\partial\rho}{\partial t} = \dfrac{\phi'(\rho)\sqrt{1+\rho_{\theta}^2\phi(\rho)^{-2}}-\phi(\rho)}{\kappa-1}-\phi(\rho),& (\theta,t)\in\mathbb{S}^1\times [0,T),\\
	\rho(\theta,0)=\rho_0(\theta).
	\end{array}\right.
	\label{equ-scalaflow}
\end{equation}
In the following, we  will use (\ref{equ-scalaflow}) to obtain some estimates.

\begin{proposition}[$C^0$ Estimate]
	\label{prop-c0est}
	Let $\rho(\cdot,t)\in\mathbb{S}^1\times[0,T)$ be the solution to \eqref{equ-scalaflow}, then
	\begin{equation}
		\min\limits_{\mathbb{S}^1}\rho(\cdot,0)\leqslant\rho(\cdot,t) \leqslant \max\limits_{\mathbb{S}^1}\rho(\cdot,0).
	\end{equation}
\end{proposition}
\begin{proof}
	At the spatial extremal point of $\rho$, we have $\rho_{\theta}=0$. The equation (\ref{equ-scalaflow}) then becomes
	\begin{equation}
		\dfrac{\partial\rho}{\partial t} = \dfrac{\phi'(\rho)-\phi(\rho)}{\kappa-1}-\phi(\rho).
	\end{equation}
	At the maximum point of $\rho$, i.e. $\rho(\theta,t)=\max\limits_{\mathbb{S}^1}\rho(\cdot,t)$, we have $\kappa\geqslant \dfrac{\phi'(\rho)}{\phi(\rho)}$. Thus
	\begin{eqnarray}
		\partial_t \rho_{\max}(t) \leqslant \dfrac{\phi'-\phi}{\frac{\phi'}{\phi}-1}-\phi =0.
	\end{eqnarray}
	At the minimum point of $\rho$, i.e. $\rho(\theta,t)=\min\limits_{\mathbb{S}^1}\rho(\cdot,t)$, we have $\kappa\leqslant \dfrac{\phi'(\rho)}{\phi(\rho)}$. Thus
	\begin{eqnarray}
		\partial_t \rho_{\min}(t) \geqslant \dfrac{\phi'-\phi}{\frac{\phi'}{\phi}-1}-\phi =0.
	\end{eqnarray}
    The conclusion follows by the standard maximum principle.
\end{proof}

\begin{proposition}[$C^1$ Estimate]
	Along the flow \eqref{equ-scalaflow}, there exists a constant $C>0$ depending only on $\rho_0$, such that 
	\begin{equation}
		\|\rho_{\theta}\|_{C^0(\mathbb{S}^1\times[0,T))}\leqslant C.
	\end{equation}
\end{proposition}
\begin{proof}
	1. Firstly, we prove that star-shapedness preserves along the flow (\ref{equ-flow}). For any $t\in[0,T)$, at the spatial minimum point $P$ of $u$, we have
	\begin{equation}
		\dfrac{\partial u}{\partial s}=\kappa\langle V,T \rangle=0.
	\end{equation}
	Since $\gamma_t$ is strictly h-convex, that is $\kappa>1$, then $\langle V,T \rangle=0$. This implies that $\dfrac{\partial}{\partial\rho}$ is parallel to the unit outward normal vector $\nu$ at $P$. Hence
	\begin{equation}
		u(P)=\langle V,\nu \rangle (P) =\phi(\rho)|_{P}.
	\end{equation}
	By Proposition \ref{prop-c0est}, we have
	\begin{equation}
		u \geqslant \min_{\gamma_t}\phi(\rho)>0. \label{equ-lowsupp}
	\end{equation}
	
	2. Combining (\ref{equ-support}) and (\ref{equ-lowsupp}), we get
	\begin{equation}
		\dfrac{|\rho_{\theta}|}{\phi(\rho)} \leqslant \sqrt{1+\rho_{\theta}^2\phi(\rho)^{-2}}\leqslant \dfrac{\max\limits_{\gamma_t}\phi(\rho)}{\min\limits_{\gamma_t}\phi(\rho)} \leqslant \dfrac{\max\limits_{\gamma_0}\phi}{\min\limits_{\gamma_0}\phi},
	\end{equation}
	where we have used that $\phi(\rho)=\sinh\rho$ is monotonically increasing in the last inequality. Therefore,
	\begin{equation}
		|\rho_{\theta}| \leqslant \dfrac{\max\limits_{\gamma_0}\phi^2}{\min\limits_{\gamma_0}\phi}:=C,
	\end{equation}
	where the constant $C$ depends only on $\rho_0$.
\end{proof}

\subsection{Long time existence}
By the above a priori estimates, we know that if $\rho$ is a positive admissible solution to the flow  (\ref{equ-scalaflow}) in $\mathbb{S}^1\times[0,T)$, then by the expression (\ref{equ-ka}) of $\kappa$, we have
\begin{equation}
	\|\rho\|_{C^2(\mathbb{S}^1\times [0,T))}\leqslant C,
\end{equation}
where $C$ depends only on $\rho_0$. By Krylov and Safanov's theorem \cite[Section 5.5]{nonlinearparabolic}, we get $\|\rho\|_{C^{2,\alpha}}\leqslant C$, where $C$ depends only on $\rho_0$. By Schauder's estimate, we get the uniform $C^k$ estimate of $\rho$ ($k\geqslant 2$). Therefore, the flow (\ref{equ-flow}) exists in $t\in(0,\infty)$.

\subsection{Smooth convergence}
We first obtain subsequential convergence.
\begin{lemma}
	There exists a sequence $\{t_j\}$ such that when $t_j\to\infty$, $Q(t_j)$ monotonically decreasing converges to 0, where
	\begin{equation}
		Q(t) = \int_{\gamma_t} \left( \dfrac{\phi'(r)-u}{\kappa -1} - u \right) \mathrm{d}s,
	\end{equation}
	and $\gamma_{t_j}$ converges smoothly to a geodesic circle centered at the origin.
\end{lemma}
\begin{proof}
	1. Existence of $\gamma_{\infty}$. Note that $A'(t)=Q(t)$, then
	\begin{equation}
		\int_0^t Q(\tau) \mathrm{d}\tau = A(t) - A(0).
	\end{equation}
	By Heintze-Karcher type inequality \eqref{equ-heintze-karcher-type}, we know $Q(t)\geqslant 0$. By Proposition \ref{prop-c0est}, $A(t)\leqslant 2\pi \Phi(\max\limits_{\mathbb{S}^1} \rho_0)$. Thus $Q(t)$ is a nonnegative  $L^1(0,\infty)$ function. Then by Arzel\`a-Ascoli theorem and uniform estimates, there exists $\{t_j\}$ such that $\gamma_{t_j}$ converges smoothly to $\gamma_{\infty}$, and the following holds:
	\begin{equation}
		\int_{\gamma_{\infty}} \left( \dfrac{\phi'-u}{\kappa-1}-u \right) \mathrm{d}s =0.
	\end{equation}
	By the rigidity of Heintze-Karcher type inequality proved in \cite[Proposition 8.2]{Li2022Hyperbolic}, we know that $\gamma_{\infty}$ is a geodesic circle.
	
	2. We prove that $\gamma_{\infty}$ is a geodesic circle centered at the origin. Let $\rho_{\infty}$ be the radius of $\gamma_{\infty}$, and its curvature is $\dfrac{\phi'(\rho_{\infty})}{\phi(\rho_{\infty})}$. We argue by contradiction. Suppose $\gamma_{\infty}$ is not centered at the origin, then at the spatial point where $\rho$ attains its maximum of $\theta$,
	\begin{equation}
		\dfrac{\partial\rho}{\partial t} = \dfrac{\phi'(\rho)-\phi(\rho)}{\dfrac{\phi'(\rho_{\infty})}{\phi(\rho_{\infty})}-1}-\phi(\rho) = \dfrac{\phi(\rho_{\infty}-\rho)}{\phi'(\rho_{\infty})-\phi(\rho_{\infty})}<0,
	\end{equation}
	which contradicts that $\gamma_{\infty}$ is the stationary solution to the flow (\ref{equ-flow}).
\end{proof}

Then we get smooth convergence.
\begin{proposition}
	There exists $\rho_{\infty}>0$ such that along the flow \eqref{equ-scalaflow}, when $t\to\infty$, there holds $\|\rho(\cdot,t)-\rho_{\infty}\|_{C^k(\mathbb{S}^1)}\to 0$, $\forall\ k\in\mathbb{N}$.
\end{proposition}
\begin{proof}
	We have proved that there exist $t_j\to\infty$ and $\rho_{\infty}$ such that $\|\rho(\cdot,t_j)-\rho_{\infty}\|_{C^k(\mathbb{S}^1)}\to 0,\ \forall\ k\in\mathbb{N}$. Since $\rho_{\max}(t)=\max\limits_{\mathbb{S}^1}\rho(\theta,t)$ is monotonically non-increasing, and $\rho_{\min}(t)=\min\limits_{\mathbb{S}^1}\rho(\theta,t)$ is monotonically non-decreasing, we have $\|\rho(\cdot,t)-\rho_{\infty}\|_{C^0(\mathbb{S}^1)}\to 0\ (t\to\infty)$.
	
	We recall the special case of Gagliardo-Nirenberg interpolation inequality \cite{Nirenberg1966AnEI}: For any positive integer $j$ and $k$, where $j<k$, and let $\delta=\frac{j}{k}$, there exists a constant $C(j,k)$, such that for any $f\in W^{k,2}(\mathbb{S}^1)$,
	\begin{equation}
		\left\|\dfrac{\partial^j}{\partial\theta^j} f\right\|_{L^2(\mathbb{S}^1)} \leqslant \left\|\dfrac{\partial^k}{\partial\theta^k} f\right\|_{L^2(\mathbb{S}^1)}^{\delta}\|f\|_{L^2(\mathbb{S}^1)}^{1-\delta}.
		\label{equ-gagli}
	\end{equation}
	Choose $f=\rho(\theta,t)-\rho_{\infty}$, we have
	\begin{equation}
		\left\|\dfrac{\partial^j}{\partial\theta^j} \rho(\theta,t)\right\|_{L^2(\mathbb{S}^1)} \leqslant \left\|\dfrac{\partial^k}{\partial\theta^k} \rho(\theta,t)\right\|_{L^2(\mathbb{S}^1)}^{\delta}\|\rho(\theta,t)-\rho_{\infty}\|_{L^2(\mathbb{S}^1)}^{1-\delta}.
	\end{equation}
	For any positive integer $j$, and any $\varepsilon>0$, for a large enough time $t$, by the boundedness of $\|\rho\|_{C^k(\mathbb{S}^1)}$ and $\|\rho(\cdot,t)-\rho_{\infty}\|_{C^0(\mathbb{S}^1)}<\varepsilon$, we know that $\left\| \dfrac{\partial^j}{\partial\theta^j} \rho(\theta,t) \right\|_{L^2(\mathbb{S}^1)}<C\varepsilon$. Then Sobolev embedding theorem implies that $\|\rho(\cdot,t)-\rho_{\infty}\|_{C^{j-1}(\mathbb{S}^1)}<C\varepsilon$. This shows that $\rho(\cdot,t)$ converges smoothly to $\rho_{\infty}$ when $t\to\infty$.
\end{proof}

\subsection{Exponential convergence}
We have shown that when $t\to\infty$, $\gamma_t$ converges smoothly to a geodesic circle centered at the origin with radius $\rho_{\infty}$. Therefore, for each large enough time $t$, $\gamma_t$ lies in the neighborhood of $\gamma_{\infty}$. We consider the linearization of flow (\ref{equ-scalaflow}) at $\rho=\rho_{\infty}$ (See \cite[Lemma 4.3]{twodim2022}). Let
\begin{equation}
	\sigma(\theta,t)=\rho(\theta,t)-\rho_{\infty},
\end{equation}
then as $t\to\infty$, $\sigma(\theta,t)\to 0$. Denote $G[\rho(\theta,t)]$ as the right hand side of (\ref{equ-scalaflow}), then
\begin{eqnarray}
	\dfrac{\partial}{\partial t} \sigma(\theta,t) &=& \dfrac{\partial}{\partial t}(\rho(\theta,t)-\rho_{\infty}) = G[\rho(\theta,t)] \notag\\
	&=& G[\rho_{\infty}] + DG|_{\rho_{\infty}}(\sigma(\theta,t)) + O\left(\|\sigma(\theta,t)\|_{C^2(\mathbb{S}^1)}^2\right) \notag\\
	&=& \dfrac{1}{\phi'(\rho_{\infty})-\phi(\rho_{\infty})}\dfrac{\partial^2}{\partial\theta^2}\sigma(\theta,t) + O\left(\|\sigma(\theta,t)\|_{C^2(\mathbb{S}^1)}^2\right).
\end{eqnarray}

Consider the eigenfunction of laplacian on $\mathbb{S}^1$ (see also \cite[Section 9]{Wei2022AFN}), i.e.
\begin{equation}
	\dfrac{\partial^2 \phi_k(\theta)}{\partial\theta^2} = -\lambda_k\phi_k(\theta),
\end{equation}
where $0=\lambda_0<\lambda_1\leqslant \cdots \leqslant \lambda_k \to \infty$, $\theta\in\mathbb{S}^1$, and $\{\phi_k(\theta)\}$ constitutes a set of orthonormal basis in $L^2(\mathbb{S}^1)$. Let
\begin{equation}
	\sigma(\theta,t) = \sum\limits_{k=0}^{\infty} \varphi_k(\theta,t),
\end{equation}
where $\varphi_k(\theta,t)=f_k(t)\phi_k(\theta)$, $f_k(t)\in C^{\infty}(I)$, then
\begin{equation}
	\varphi_0 = \dfrac{1}{2\pi}\int_{\mathbb{S}^1} \sigma(\theta,t) \mathrm{d}\theta.
\end{equation}
Therefore,
\begin{eqnarray}
	\dfrac{\partial}{\partial t}\int_{\mathbb{S}^1} \sigma(\theta,t)^2 \mathrm{d}\theta &=& \int_{\mathbb{S}^1} 2\sigma(\theta,t)\dfrac{\partial}{\partial t}\sigma(\theta,t) \mathrm{d}\theta\notag\\
	&\leqslant& \dfrac{2}{\phi'(\rho_{\infty})-\phi(\rho_{\infty})} \int_{\mathbb{S}^1} \sigma(\theta,t) \dfrac{\partial^2}{\partial\theta^2} \sigma(\theta,t) \mathrm{d}\theta + C\|\sigma(\theta,t)\|_{C^2(\mathbb{S}^1)}^2 \int_{\mathbb{S}^1} |\sigma(\theta,t)|\mathrm{d}\theta\notag\\
	&=&\dfrac{2}{\phi'(\rho_{\infty})-\phi(\rho_{\infty})}\int_{\mathbb{S}^1} -\sum\limits_{k\geqslant 1} \lambda_k\varphi_k^2(\theta,t)\mathrm{d}\theta + C \|\sigma(\theta,t)\|^2_{C^2(\mathbb{S}^1)}\int_{\mathbb{S}^1}|\sigma(\theta,t)|\mathrm{d}\theta \notag\\
	&=& -\dfrac{2}{\phi'(\rho_{\infty})-\phi(\rho_{\infty})}\int_{\mathbb{S}^1} \sum\limits_{k\geqslant 1} \lambda_1 \varphi_k^2(\theta,t) \mathrm{d}\theta \notag\\
	&&- \dfrac{2}{\phi'(\rho_{\infty})-\phi(\rho_{\infty})}\int_{\mathbb{S}^1} \sum\limits_{k\geqslant 1} (\lambda_k-\lambda_1)\varphi_k^2(\theta,t) \mathrm{d}\theta \notag\\
	&&+C \|\sigma(\theta,t)\|^2_{C^2(\mathbb{S}^1)}\int_{\mathbb{S}^1}|\sigma(\theta,t)|\mathrm{d}\theta \notag\\
	&\leqslant& -\dfrac{2\lambda_1}{\phi'(\rho_{\infty})-\phi(\rho_{\infty})}\int_{\mathbb{S}^1} \sum\limits_{k\geqslant 0} \varphi_k^2(\theta,t) \mathrm{d}\theta + \dfrac{2\lambda_1}{\phi'(\rho_{\infty})-\phi(\rho_{\infty})}\int_{\mathbb{S}^1} \varphi_0^2(\theta,t) \mathrm{d}\theta \notag\\
	&&+C\|\sigma(\theta,t)\|^2_{C^2(\mathbb{S}^1)}\int_{\mathbb{S}^1}|\sigma(\theta,t)|\mathrm{d}\theta \notag\\
	&=& -\dfrac{2\lambda_1}{\phi'(\rho_{\infty})-\phi(\rho_{\infty})}\int_{\mathbb{S}^1} \sigma(\theta,t)^2 \mathrm{d}\theta + \underbrace{ \dfrac{\lambda_1}{\pi \left(\phi'(\rho_{\infty})-\phi(\rho_{\infty})\right)}\left( \int_{\mathbb{S}^1}\sigma(\theta,t) \mathrm{d}\theta \right)^2 }_{I}\notag\\
	&&+ \underbrace{ C\|\sigma(\theta,t)\|^2_{C^2(\mathbb{S}^1)}\int_{\mathbb{S}^1}|\sigma(\theta,t)|\mathrm{d}\theta}_{II}. \label{equ-dtsigma}
\end{eqnarray}

By Lemma \ref{lemma-la}, we know that $[L-A](\rho(\theta,t))=[L-A](\rho_{\infty})$. We claim that: There exists a constant $C$ such that
\begin{equation}
	\left( \int_{\mathbb{S}^1}\sigma(\theta,t) \mathrm{d}\theta \right)^2 \leqslant C \left( \int_{\mathbb{S}^1} \|\sigma(\theta,t)\|^2_{C^2(\mathbb{S}^1)}  \mathrm{d}\theta \right)^2 \leqslant C\|\sigma(\theta,t)\|_{C^2(\mathbb{S}^1)}^4. \label{equ-integralsigma2}
\end{equation}
Indeed, we use the Taylor expansion of $[L-A](\rho(\theta,t))=[L-A](\sigma(\theta,t)+\rho_{\infty})$ at $\rho_{\infty}$ to obtain
\begin{equation}
	[L-A](\rho(\theta,t)) = [L-A](\rho_{\infty}) + \left( \phi'(\rho_{\infty})-\phi(\rho_{\infty}) \right)\int_{\mathbb{S}^1} \sigma(\theta,t) \mathrm{d}\theta + \int_{\mathbb{S}^1} O\left( \|\sigma(\theta,t)\|_{C^2(\mathbb{S}^1)}^2 \right) \mathrm{d}\theta.
\end{equation}
This implies
\begin{equation}
	\left( \phi'(\rho_{\infty})-\phi(\rho_{\infty}) \right)\int_{\mathbb{S}^1} \sigma(\theta,t) \mathrm{d}\theta + \int_{\mathbb{S}^1} O\left( \|\sigma(\theta,t)\|_{C^2(\mathbb{S}^1)}^2 \right) \mathrm{d}\theta = 0,
\end{equation}
and thus there exists a constant $C$ such that
\begin{equation*}
	\left( \int_{\mathbb{S}^1}\sigma(\theta,t) \mathrm{d}\theta \right)^2 \leqslant C \left( \int_{\mathbb{S}^1} \|\sigma(\theta,t)\|^2_{C^2(\mathbb{S}^1)}  \mathrm{d}\theta \right)^2 \leqslant C\|\sigma(\theta,t)\|_{C^2(\mathbb{S}^1)}^4.
\end{equation*}

In the following, we will choose $f=\sigma(\theta,t)$ in Gagliardo-Nirenberg interpolation inequality (\ref{equ-gagli}) to estimate $\|\sigma(\theta,t)\|_{C^2(\mathbb{S}^1)}^2$. For any positive integer $j,\ k$ satisfying $j<k$, let $\delta=\frac{j}{k}$. There exists a constant $C(j,k)$ such that
\begin{equation}
	\left\| \dfrac{\partial^j}{\partial\theta^j} \sigma(\theta,t) \right\|_{L^2(\mathbb{S}^1)} \leqslant \left\| \dfrac{\partial^k}{\partial\theta^k} \sigma(\theta,t) \right\|_{L^2(\mathbb{S}^1)}^{\delta} \|\sigma(\theta,t)\|_{L^2(\mathbb{S}^1)}^{1-\delta}.  \label{equ-sigmaf}
\end{equation}
By Sobolev embedding theorem, we have
\begin{equation}
	\|\sigma(\theta,t)\|_{C^2(\mathbb{S}^1)} \leqslant C \|\sigma(\theta,t)\|_{W^{j,2}(\mathbb{S}^1)},\ \ j\geqslant 3.  \label{equ-sigmasobo}
\end{equation}
Combining (\ref{equ-sigmaf}) and (\ref{equ-sigmasobo}), taking $k$ large enough, we obtain that $\forall\ 0<\varepsilon<1$, there exists a positive constant $C(\varepsilon)$ such that
\begin{equation}
	\|\sigma(\theta,t)\|_{C^2(\mathbb{S}^1)}^2 \leqslant C(\varepsilon) \|\sigma(\theta,t)\|_{L^2(\mathbb{S}^1)}^{1+\varepsilon}.
\end{equation}
Choose $\varepsilon=\dfrac{1}{2}$, we get
\begin{equation}
	\|\sigma(\theta,t)\|_{C^2(\mathbb{S}^1)}^2 \leqslant C \|\sigma(\theta,t)\|_{L^2(\mathbb{S}^1)}^{\frac{3}{2}}. \label{equ-sigma23}
\end{equation}

Now from (\ref{equ-integralsigma2}) and (\ref{equ-sigma23}), we know
\begin{equation}
	\left( \int_{\mathbb{S}^1} \sigma(\theta,t) \mathrm{d}\theta \right)^2 \leqslant C\|\sigma(\theta,t)\|_{L^2(\mathbb{S}^1)}^{3}=C\left( \int_{\mathbb{S}^1} \sigma(\theta,t)^2 \mathrm{d}\theta \right)^{1+\frac{1}{2}}. \label{equ-sigmgintegral}
\end{equation}
By (\ref{equ-sigma23}) and H\"{o}lder inequality, we obtain
\begin{eqnarray}
	\|\sigma(\theta,t)\|_{C^2(\mathbb{S}^1)}^2\int_{\mathbb{S}^1} |\sigma(\theta,t)|\mathrm{d}\theta \leqslant C\|\sigma(\theta,t)\|_{L^2(\mathbb{S}^1)}^{\frac{5}{2}}=C\left( \int_{\mathbb{S}^1} \sigma(\theta,t)^2 \mathrm{d}\theta \right)^{1+\frac{1}{4}}. \label{equ-sigmac2}
\end{eqnarray}
Substituting (\ref{equ-sigmgintegral}) and (\ref{equ-sigmac2}) into (\ref{equ-dtsigma}), we get
\begin{eqnarray}
	&&\dfrac{\partial}{\partial t} \int_{\mathbb{S}^1}\sigma(\theta,t)^2\mathrm{d}\theta\notag\\
	&=& -\dfrac{2\lambda_1}{\phi'(\rho_{\infty})-\phi(\rho_{\infty})}\int_{\mathbb{S}^1} \sigma(\theta,t)^2 \mathrm{d}\theta + C\left( \int_{\mathbb{S}^1} \sigma(\theta,t)^2 \mathrm{d}\theta \right)^{1+\frac{1}{2}} + C\left( \int_{\mathbb{S}^1} \sigma(\theta,t)^2 \mathrm{d}\theta \right)^{1+\frac{1}{4}}\notag\\
	&=& -\dfrac{2\lambda_1}{\phi'(\rho_{\infty})-\phi(\rho_{\infty})}\int_{\mathbb{S}^1} \sigma(\theta,t)^2 \mathrm{d}\theta \left( 1-C'\left( \int_{\mathbb{S}^1} \sigma(\theta,t)^2 \mathrm{d}\theta \right)^{\frac{1}{2}} -C''\left( \int_{\mathbb{S}^1} \sigma(\theta,t)^2 \mathrm{d}\theta \right)^{\frac{1}{4}} \right). \notag
\end{eqnarray}
Since $\int_{\mathbb{S}^1}\sigma(\theta,t)^2 \mathrm{d}\theta\to 0 \ (t\to\infty)$, for any $\delta>0$, there exists $t_{\delta}>0$ such that for any $t>t_{\delta}$,
\begin{eqnarray}
	\dfrac{\partial}{\partial t} \int_{\mathbb{S}^1}\sigma(\theta,t)^2 \mathrm{d}\theta \leqslant  -\dfrac{2\lambda_1}{\phi'(\rho_{\infty})-\phi(\rho_{\infty})}\int_{\mathbb{S}^1} \sigma(\theta,t)^2 \mathrm{d}\theta \left( 1-\delta \right).
\end{eqnarray}
Choose $\delta=\frac{1}{2}$, we obtain
\begin{equation}
	\int_{\mathbb{S}^1} \sigma(\theta,t)^2 \mathrm{d}\theta \leqslant C\mathrm{e}^{-\frac{\lambda_1}{\phi'(\rho_{\infty})-\phi(\rho_{\infty})}t}.
\end{equation}
By (\ref{equ-sigmaf}) and Sobolev embedding theorem:
\begin{equation}
	\|\sigma(\theta,t)\|_{C^{j-1}(\mathbb{S}^1)}\leqslant C\|\sigma(\theta,t)\|_{W^{j,2}(\mathbb{S}^1)},\ \ j\geqslant 1,
\end{equation}
we obtain that
\begin{equation}
	\|\sigma(\theta,t)\|_{C^k(\mathbb{S}^1)} \leqslant C_1\mathrm{e}^{-C_2 t},\ \ \forall\ k\geqslant 0,
\end{equation}
where $C_1,\ C_2$ depends only on $\rho_0,\ k$.

Above all, we complete the proof of Theorem \ref{thm-existenceandconvergence}.

\section{Geometric inequalities}
\label{section-4}
In this section, we will prove geometric inequalities involving weighted curvature integrals and modified quermassintegrals for h-convex domains.
\subsection{Inequalities for curves}
We will prove Theorem \ref{thm-inequality2dim} in this subsection.
\begin{lemma}
	\label{lemma-mono}
	Along the flow \eqref{equ-flow}, if the initial curve is strictly h-convex, then there holds
	\begin{equation}
		\dfrac{\partial}{\partial t} \int_{\gamma_t} \left( \Phi-u \right)(\kappa-1) \mathrm{d}s \leqslant 0 \label{equ-mono}
	\end{equation}
	with equality if and only if $\gamma_t$ is a geodesic circle centered at the origin.
\end{lemma}
\begin{proof}
	Along the flow $\partial_t X=\eta\nu$, we calculate that
	\begin{eqnarray}
		\dfrac{\partial}{\partial t} \int_{\gamma_t} (\Phi-u)(\kappa-1) \mathrm{d}s &=& \int_{\gamma_t} \left(\partial_t\Phi-\partial_t u\right)(\kappa-1) \mathrm{d}s + \int_{\gamma_t} (\Phi-u)\partial_t \kappa \mathrm{d}s \notag\\
		&&+ \int_{\gamma_t} (\Phi-u)(\kappa-1) \partial_t \mathrm{d}s.
	\end{eqnarray}
	Note that
	\begin{eqnarray}
		\partial_t \Phi &=& \langle \overline{\nabla}\Phi, \partial_t X \rangle = \langle V, \eta\nu \rangle = \eta u,\\
		\partial_t u &=& \langle \overline{\nabla}_{\partial_t X} V,\nu \rangle + \langle V,\partial_t\nu \rangle \notag\\
		&=& \eta \langle \overline{\nabla}_{\nu} V,\nu \rangle - \langle V,\dfrac{\partial \eta}{\partial s}T \rangle\notag\\
		&=& \eta \phi' -\dfrac{\partial \eta}{\partial s}\dfrac{\partial\Phi}{\partial s},
	\end{eqnarray}
	we get the evolution equation
	\begin{eqnarray}
		\partial_t \int_{\gamma_t} (\Phi-u)(\kappa-1) \mathrm{d}s &=& \int_{\gamma_t} \left[(u-\phi')\eta + \dfrac{\partial \eta }{\partial s}\dfrac{\partial\Phi}{\partial s} \right](\kappa-1) \mathrm{d}s \notag\\
		&&-\int_{\gamma_t} (\Phi-u)\left[ \dfrac{\partial^2 \eta }{\partial s^2} + (\kappa^2-1) \eta  \right] \mathrm{d}s \notag\\
		&&+ \int_{\gamma_t} (\Phi-u)(\kappa-1)\kappa \eta \mathrm{d}s.
	\end{eqnarray}
	Using integration by parts, we have
	\begin{eqnarray}
		\int_{\gamma_t} (\Phi-u)\dfrac{\partial^2 \eta }{\partial s^2} \mathrm{d}s = -\int_{\gamma_t} \left(\dfrac{\partial \Phi}{\partial s} - \dfrac{\partial u}{\partial s}\right) \dfrac{\partial \eta }{\partial s} \mathrm{d}s = \int_{\gamma_t} (\kappa-1)\dfrac{\partial\Phi}{\partial s}\dfrac{\partial \eta }{\partial s} \mathrm{d}s.
	\end{eqnarray}
	We obtain that
	\begin{eqnarray}
		\partial_t \int_{\gamma_t} (\Phi-u)(\kappa-1) \mathrm{d}s &=& \int_{\gamma_t} (u-\phi')(\kappa-1)\eta  \mathrm{d}s - \int_{\gamma_t} (\Phi-u)(\kappa-1)\eta  \mathrm{d}s \notag\\
		&=& \int_{\gamma_t} (2u-2\Phi-1)(\kappa-1)\eta  \mathrm{d}s, \label{equ-phiu}
	\end{eqnarray}
	where we used $\phi'=\Phi+1$. Substituting
	\begin{equation*}
		\eta =\dfrac{\phi'-u}{\kappa-1}-u=\dfrac{\phi'-\kappa u}{\kappa-1}
	\end{equation*}
    into (\ref{equ-phiu}), using \eqref{equ-Phi'} and integration by parts, we obtain that
	\begin{eqnarray}
		\partial_t \int_{\gamma_t} (\Phi-u)(\kappa-1) \mathrm{d}s &=& \int_{\gamma_t} (2u-2\Phi-1)(\phi'-\kappa u) \mathrm{d}s\notag\\
		&=& \int_{\gamma_t} (2u-2\Phi-1)\dfrac{\partial^2 \Phi}{\partial s^2} \mathrm{d}s \notag\\
		&=& - 2\int_{\gamma_t} \left( \dfrac{\partial u}{\partial s}-\dfrac{\partial\Phi}{\partial s} \right) \dfrac{\partial\Phi}{\partial s} \mathrm{d}s \notag\\
		&=& -2\int_{\gamma_t} (\kappa-1) \left| \dfrac{\partial\Phi}{\partial s} \right|^2 \mathrm{d}s \leqslant 0, \label{equ-phiueq0}
	\end{eqnarray}
	where we used that the curvature of $\gamma_t$ is $\kappa>1$ by Proposition \ref{prop-curvature}. Equality holds in (\ref{equ-phiueq0}) if and only if $\dfrac{\partial\Phi}{\partial s}\equiv 0$ holds along $\gamma_t$, which is equivalent to that $r\equiv r_0$. Thus $\gamma_t$ is a geodesic circle centered at the origin.
\end{proof}

\begin{proof}[Proof of Theorem \ref{thm-inequality2dim}]
	We divide the proof into two steps.
	
	{\bf Step 1.} Suppose $\gamma$ is a smooth strictly h-convex closed curve in $\mathbb{H}^2$. Let $\gamma$ be the initial value of the flow (\ref{equ-flow}) and we obtain the solution $\{\gamma_t\}_{t\in[0,\infty)}$ by Theorem \ref{thm-existenceandconvergence}. Moreover, when $t\to\infty$, $\gamma_t$ converges smoothly to a geodesic circle $\gamma_{\infty}$ centered at the origin with radius $\rho_{\infty}$. By (\ref{equ-mono1}) and (\ref{equ-mono}), we have
	\begin{eqnarray}
		\int_{\gamma}(\Phi-u)(\kappa-1)\mathrm{d}s + (L[\gamma]-A[\gamma]) &\geqslant& \int_{\gamma_{\infty}}(\Phi-u)(\kappa-1)\mathrm{d}s + (L[\gamma_{\infty}]-A[\gamma_{\infty}])\notag\\
		&=& 2\pi \left(\Phi(\rho_{\infty})-\phi(\rho_{\infty})\right)\left( \phi'(\rho_{\infty})-\phi(\rho_{\infty}) \right) \notag\\
		&&+ 2\pi \left( \phi(\rho_{\infty})-\Phi(\rho_{\infty}) \right)\notag\\
		(\mbox{by}\ \phi'=\Phi+1)	&=&  2\pi \left( \phi(\rho_{\infty})-\Phi(\rho_{\infty}) \right)^2 \notag\\
		&=& \dfrac{1}{2\pi} \left( L[\gamma_{\infty}]-A[\gamma_{\infty}] \right)^2 \notag\\
		&=& \dfrac{1}{2\pi} \left( L[\gamma]-A[\gamma] \right)^2. \label{equ-step1curve}
	\end{eqnarray}
	By Lemma \ref{lemma-mono}, equality holds in \eqref{equ-step1curve} if and only if $\gamma$ is a geodesic circle centered at the origin.
	
	{\bf Step 2.} Now we consider $\gamma$ as a smooth h-convex closed curve in $\mathbb{H}^2$. We evolve $\gamma$ along curve shortening flow
	\begin{equation}
		\partial_t X = -\kappa\nu,
	\end{equation}
	and obtain solution curve $\gamma_t$ with curvature $\kappa(x,t)$ satisfying
	\begin{equation}
		\dfrac{\partial}{\partial t} (\kappa-1) = \dfrac{\partial^2}{\partial s^2}(\kappa-1) + (\kappa-1)\kappa(\kappa+1).
	\end{equation}
	Note that for a compact h-convex curve, there exists at least one point on $\gamma$ with $\kappa>1$. By strong maximum principle, $\gamma_t\ (\forall\ t>0)$ is strictly h-convex. Note that for $\gamma_t\ (t>0)$, inequality (\ref{equ-ineq}) holds by {\bf Step 1}. Now let $t\to 0$, we obtain that inequality (\ref{equ-ineq}) holds for the h-convex curve $\gamma$.
	
	If equality in (\ref{equ-ineq}) holds for a h-convex curve $\gamma$, we are going to show that $\gamma$ is strictly h-convex. Let
	\begin{equation}
		\tilde{\gamma} = \{ p\in\gamma:\kappa(p)>1 \}.
	\end{equation}
	Since there exists at least one point on $\gamma$ satisfying $\kappa>1$ for a compact h-convex curve $\gamma$, we know that $\tilde{\gamma}$ is a nonempty open set in $\gamma$. We next show that $\tilde{\gamma}$ is also closed in $\gamma$. Following \cite{Guan2009TheQI}, we choose a smooth function $\varphi\in C_c^{\infty}(\tilde{\gamma})$ with  compact support in $\tilde{\gamma}$, and consider a normal variation of $\gamma$
	\begin{equation}
		\partial_t X= -\varphi\nu,
	\end{equation}
	where $t\in(-\varepsilon,\varepsilon)$, and $\varepsilon>0$ is small enough such that $\gamma_t=X(\mathbb{S}^1,t)$ is h-convex. Let
	\begin{equation}
		Q(t) = \int_{\gamma_t} (\Phi-u)(\kappa-1) \mathrm{d}s + (L[\gamma_t]-A[\gamma_t]) - \dfrac{1}{2\pi}(L[\gamma_t]-A[\gamma_t])^2,
	\end{equation}
	we have
	\begin{equation}
		Q(t)\geqslant 0,\ t\in(-\varepsilon,\varepsilon),\ \ Q(0)=0,
	\end{equation}
	which implies that $Q(t)\ (t\in(-\varepsilon,\varepsilon))$ attains its minimum at $t=0$. Thus
	\begin{eqnarray}
		0 &=& \left.\dfrac{\mathrm{d}}{\mathrm{d}t}\right|_{t=0} Q(t) \notag\\
		&=& -\int_{\gamma} (2u-2\Phi-1)(\kappa-1)\varphi \mathrm{d}s-\int_{\gamma} (\kappa-1)\varphi \mathrm{d}s + \dfrac{L[\gamma]-A[\gamma]}{\pi}\int_{\gamma} (\kappa-1)\varphi \mathrm{d}s \notag\\
		&=& \int_{\gamma} \left( \dfrac{L[\gamma]-A[\gamma]}{\pi}-(2u-2\Phi) \right)(\kappa-1) \varphi \mathrm{d}s, \ \  \forall\  \varphi\in C_c^{\infty}(\tilde{\gamma}).
	\end{eqnarray}
	In particular, the curve $\tilde{\gamma}$ satisfies
	\begin{equation}
		2u-2\Phi\equiv \dfrac{L[\gamma]-A[\gamma]}{\pi},
	\end{equation}
	and hence
	\begin{equation}
		\tilde{\gamma}=\left\lbrace p\in\gamma:(u-\Phi)(p)=\dfrac{L[\gamma]-A[\gamma]}{2\pi} \right\rbrace.
	\end{equation}
	This implies that $\tilde{\gamma}$ is closed in $\gamma$. By the connectness of $\gamma$, $\gamma\equiv\tilde{\gamma}$, thus $\gamma$ is strictly h-convex. By the equality case in {\bf Step 1}, we conclude that $\gamma$ is a geodesic circle centered at the origin.
\end{proof}

\subsection{Inequalities for hypersurfaces}
We include the complete proof of Theorem \ref{thm-inequalityhighdim} in this subsection though it has been proved in \cite{Li2022Hyperbolic}. Some details in this proof will be used in proving stability results later. Let's first derive the general evolution equation.
\begin{lemma}
	Along the flow $\partial_t X=\eta\nu$, hypersurface $M_t=\partial\Omega_t\subset\mathbb{H}^{n+1}$ $(n\geqslant 2)$ satisfies
	\begin{eqnarray}
		\dfrac{\partial}{\partial t} \left[ \int_{M_t} (\Phi-u)E_k(\tilde{\kappa}) \mathrm{d}\mu_t + (n-2k)\tilde{W}_k(\Omega_t) \right] &=& (n-k) \int_{M_t} \Phi E_{k+1}(\tilde{\kappa})\eta \mathrm{d}\mu_t \notag\\
		&&- (k+1) \int_{M_t} (\phi'-u) E_k(\tilde{\kappa}) \eta \mathrm{d}\mu_t. \label{equ-moPhiu}
	\end{eqnarray}
\end{lemma}
\begin{proof}
	Denote $S^i_j=h^i_j-\delta^i_j$, we will compute
	\begin{eqnarray}
		&&\dfrac{\partial}{\partial t} \left( \int_{M_t} (\Phi-u)E_k(\tilde{\kappa}) \mathrm{d}\mu_t + (n-2k)\tilde{W}_k(\Omega_t) \right)\notag\\ 
		&=& \int_{M_t} (\partial_t \Phi- \partial_t u) E_k(\tilde{\kappa}) \mathrm{d}\mu_t + \int_{M_t} (\Phi-u) \partial_t E_k(\tilde{\kappa}) \mathrm{d}\mu_t \notag\\
		&&+ \int_{M_t} (\Phi-u)E_k(\tilde{\kappa}) \eta H \mathrm{d}\mu_t + (n-2k)\int_{M_t} \eta E_k(\tilde{\kappa}) \mathrm{d}\mu_t.
	\end{eqnarray}
	Substituting (\ref{equ-ptwk}), (\ref{equ-ptPhi}), (\ref{equ-ptu}) and (\ref{equ-ptek}), and combining $nE_1(\tilde{\kappa})=H-n$ and $\phi'=\Phi+1$, we obtain
	\begin{eqnarray}
		&&\dfrac{\partial}{\partial t} \left( \int_{M_t} (\Phi-u)E_k(\tilde{\kappa}) \mathrm{d}\mu_t + (n-2k)\tilde{W}_k(\Omega_t) \right)\notag\\ 
		&=& \int_{M_t} \left[(u-\phi')\eta+\langle V,\nabla \eta \rangle\right] E_k(\tilde{\kappa}) \mathrm{d}\mu_t - \int_{M_t} (\Phi-u) \dot{E}_k^{ij}(\tilde{\kappa})\nabla_j\nabla^i \eta \mathrm{d}\mu_t\notag\\
		&&+(n-2k)\int_{M_t} (\phi'-u)E_k(\tilde{\kappa}) \eta \mathrm{d}\mu_t + (n-k) \int_{M_t} (\Phi-u) E_{k+1}(\tilde{\kappa})\eta \mathrm{d}\mu_t. \label{equ-pt1}
	\end{eqnarray}
	By divergence free property of $E_k(\tilde{\kappa})$ (Lemma \ref{lem-div}) and using integration by parts, we have
	\begin{eqnarray}
		&&-\int_{M_t} (\Phi-u) \dot{E}_k^{ij}(\tilde{\kappa})\nabla_j\nabla^i \eta \mathrm{d}\mu_t \notag\\
		&=& -\int_{M_t} \dot{E}_k^{ij}(\tilde{\kappa})\nabla_j\nabla^i(\Phi-u)\eta \mathrm{d}\mu_t \notag\\
		&=& -\int_{M_t} \dot{E}_k^{ij}(\tilde{\kappa}) \left[ -\phi'S^i_j-\langle V,\nabla S^i_j \rangle + u\left( (S^2)^i_j+S^i_j \right) \right] \eta \mathrm{d}\mu_t \notag\\
		&=& \int_{M_t} \left[ \phi' k E_k(\tilde{\kappa}) + \langle V, \nabla E_k(\tilde{\kappa})  \rangle \right] \eta \mathrm{d}\mu_t \notag \\
		&&- \int_{M_t}  u \left( nE_1(\tilde{\kappa}) E_k(\tilde{\kappa}) - (n-k) E_{k+1}(\tilde{\kappa}) + kE_k(\tilde{\kappa}) \right)\eta \mathrm{d}\mu_t. \label{equ-pt2}
	\end{eqnarray}
	Note that by using integration by parts and $\Delta\Phi=n\phi'-Hu=n\phi'-(nE_1(\tilde{\kappa})+n)u$, we obtain
	\begin{eqnarray}
		&&\int_{M_t} \langle V,\nabla\eta \rangle E_k(\tilde{\kappa}) \mathrm{d}\mu_t + \int_{M_t} \langle V,\nabla E_k(\tilde{\kappa}) \rangle \eta \mathrm{d}\mu_t \notag\\
		&=& \int_{M_t} \langle V,\nabla(\eta E_k(\tilde{\kappa})) \rangle \mathrm{d}\mu_t \notag\\
		&=& -\int_{M_t} \Delta\Phi E_k(\tilde{\kappa}) \eta \mathrm{d}\mu_t \notag\\
		&=& -\int_{M_t} \left[ n\phi'- \left( nE_1(\tilde{\kappa})+n \right)u \right] E_k(\tilde{\kappa}) \eta \mathrm{d}\mu_t. \label{equ-pt3}
	\end{eqnarray}
	Combining (\ref{equ-pt1}), (\ref{equ-pt2}) and (\ref{equ-pt3}), and after a direct computation we will obtain (\ref{equ-moPhiu}).
\end{proof}

\begin{proof}[Proof of Theorem \ref{thm-inequalityhighdim}]
	For $k=1,2,\cdots,n$, we evolve $M$ along the flow (\ref{equ-flowhy}) and get a family of strictly h-convex hypersurfaces $\{M_t\}_{t\in[0,\infty)}$ by Theorem \ref{thm-higherdimension}. Moreover, as $t\to\infty$, $M_t$ converges smoothly to a geodesic sphere $M_{\infty}=\partial \Omega_{\infty}$ centered at the origin with radius $\rho_{\infty}$.
	
	Substituting $\eta = \dfrac{(\phi'-u)E_{k-1}(\tilde{\kappa})}{E_k(\tilde{\kappa})}-u$ in (\ref{equ-moPhiu}), by Newton-MacLaurin inequality (\ref{equ-MewtonMacL}) and (\ref{equ-ekdotijphi}), we have
	\begin{eqnarray}
		&&\dfrac{\partial}{\partial t} \left[ \int_{M_t} (\Phi-u)E_k(\tilde{\kappa}) \mathrm{d}\mu_t + (n-2k)\tilde{W}_k(\Omega_t) \right] \notag\\
		&\leqslant& (n-k)\int_{M_t} \Phi \left((\phi'-u)E_k(\tilde{\kappa})-uE_{k+1}(\tilde{\kappa})\right) \mathrm{d}\mu_t \notag\\
		&&- (k+1) \int_{M_t} (\phi'-u)\left( (\phi'-u) E_{k-1}(\tilde{\kappa}) - uE_k(\tilde{\kappa}) \right) \mathrm{d}\mu_t \notag\\
		&=& \dfrac{n-k}{k+1}\int_{M_t} \Phi \dot{E}_{k+1}^{ij}(\tilde{\kappa})\nabla_j\nabla^i \phi' \mathrm{d}\mu_t - \dfrac{k+1}{k} \int_{M_t} (\phi'-u) \dot{E}^{ij}_k(\tilde{\kappa}) \nabla_j\nabla^i \phi' \mathrm{d}\mu_t \notag\\
		&=& -\dfrac{n-k}{k+1} \int_{M_t} \dot{E}^{ij}_{k+1}(\tilde{\kappa}) \nabla_j\Phi \nabla^i \phi' \mathrm{d}\mu_t + \dfrac{k+1}{k} \int_{M_t} \dot{E}_k^{ij}(\tilde{\kappa}) \nabla_j(\phi'-u) \nabla^i \phi' \mathrm{d}\mu_t \notag\\
		&=& -\dfrac{n-k}{k+1} \int_{M_t} \dot{E}^{ij}_{k+1}(\tilde{\kappa}) \nabla_j\phi' \nabla^i \phi' \mathrm{d}\mu_t - \dfrac{k+1}{k} \int_{M_t} \dot{E}_k^{ij}(\tilde{\kappa}) S_j^l \nabla_l \phi' \nabla^i \phi' \mathrm{d}\mu_t. \label{equ-Mo1}
	\end{eqnarray}
	Note that $\dot{E}_{k}^{ij}(\tilde{\kappa})$, $\dot{E}_{k+1}^{ij}(\tilde{\kappa})$, $S_j^l$ are all symmetric positive definite, we conclude that
	\begin{equation}
		\dfrac{\partial}{\partial t} \left[ \int_{M_t} (\Phi-u)E_k(\tilde{\kappa}) \mathrm{d}\mu_t + (n-2k)\tilde{W}_k(\Omega_t) \right] \leqslant 0 \label{equ-Mo2}
	\end{equation}
	with equality holds if and only if $\nabla\phi'\equiv 0$, which implies that $M_t$ is a geodesic sphere centered at the origin.
	
	By Minkowski identity (\ref{equ-shiftedminkowski}), along the flow (\ref{equ-flowhy}), we have
	\begin{equation}
		\dfrac{\partial}{\partial t} \tilde{W}_k(\Omega_t) = \int_{M_t} \left[(\phi'-u) E_{k-1}(\tilde{\kappa}) - u E_k(\tilde{\kappa}) \right] \mathrm{d}\mu_t = 0.
	\end{equation}
	
	Therefore,
	\begin{eqnarray}
		\int_{M}(\Phi-u)E_k(\tilde{\kappa}) \mathrm{d}\mu + (n-2k) \tilde{W}_k(\Omega) &\geqslant& \int_{M_{\infty}}(\Phi-u)E_k(\tilde{\kappa}) \mathrm{d}\mu_{\infty} + (n-2k) \tilde{W}_k(\Omega_{\infty}) \notag\\
		&=& \tilde{h}_k\circ \tilde{f}_k^{-1} \left(\tilde{W}_k(\Omega_{\infty})\right) \notag\\
		&=& \tilde{h}_k\circ \tilde{f}_k^{-1} \left(\tilde{W}_k(\Omega)\right). \label{equ-Phi-u}
	\end{eqnarray}
	Here $\tilde{h}_k(r)= \int_{\partial\overline{B}_r}(\Phi-u)E_k(\tilde{\kappa}) \mathrm{d}\mu + (n-2k) \tilde{W}_k(\overline{B}_r)$, $\tilde{f}_k(r)=\tilde{W}_k(\overline{B}_r)$, $\overline{B}_r$ is a geodesic ball centered at the origin with radius $r$. Equality holds in \eqref{equ-Phi-u} if and only if $M_t$ is a geodesic sphere centered at the origin. Let $t\to 0$, we conclude that $M$ is a geodesic sphere centered at the origin.
\end{proof}

\section{Stability results}
\label{section-5}
In this section, we will prove our stability results Theorems \ref{thm-stabilitycurve} and \ref{thm-stabilityhypersurface}.
\subsection{Stability of inequalities for curves}
We will discuss the stability of inequality (\ref{equ-ineq}).

\begin{proof}[Proof of Theorem \ref{thm-stabilitycurve}]
	Let
	\begin{equation}
		\varepsilon := \int_{\gamma} (\Phi-u)(\kappa-1) \mathrm{d}s + (L[\gamma]-A[\gamma]) - \dfrac{1}{2\pi}(L[\gamma]-A[\gamma])^2. \label{equ-curveep}
	\end{equation}
	We evolve $\gamma$ along flow (\ref{equ-flow}) and obtain the solution curve $\{\gamma_t\}_{t\in[0,\infty)}$ by Theorem \ref{thm-existenceandconvergence}. Moreover, $\gamma_t$ is strictly h-convex and smoothly converges to a geodesic circle $\gamma_{\infty}$ centered at the origin as $t\to\infty$. Let
	\begin{equation}
		Q(t):=\int_{\gamma_t} (\Phi-u)(\kappa-1) \mathrm{d}s + (L[\gamma_t]-A[\gamma_t]).
	\end{equation}
	On the one hand, by Lemma \ref{lemma-la} and \eqref{equ-curveep}, we have
	\begin{eqnarray}
		\int_{0}^{\infty} \partial_t Q(t) \mathrm{d}t &=& Q(\infty) - Q(0) \notag\\
		&=& \dfrac{1}{2\pi}\left( L[\gamma_{\infty}]-A[\gamma_{\infty}] \right)^2 - Q(0)\notag\\
		&=& \dfrac{1}{2\pi}\left( L[\gamma]-A[\gamma] \right)^2 - Q(0)\notag\\
		&=& -\varepsilon. \label{equ-qt1}
	\end{eqnarray}
	On the other hand, by (\ref{equ-phiueq0}),
	\begin{eqnarray}
		\int_0^{\infty} \partial_t Q(t) \mathrm{d}t = -\int_0^{\infty}\int_{\gamma_t} 2(\kappa-1)\left|\dfrac{\partial\Phi}{\partial s}\right|^2\mathrm{d}s\mathrm{d}t. \label{equ-qt2}
	\end{eqnarray}
	Combining (\ref{equ-qt1}) and (\ref{equ-qt2}), we obtain
	\begin{equation}
		\int_0^{\infty}\int_{\gamma_t} 2(\kappa-1)\left| \dfrac{\partial\Phi}{\partial s} \right|^2 \mathrm{d}s\mathrm{d}t = \varepsilon.
	\end{equation}
	Since $\gamma_t$ is strictly h-convex, there exists a constant $C=C(\min\limits_{\gamma}\kappa)>0$ such that $2(\kappa-1)\geqslant C$. Thus
	\begin{equation}
		\int_0^{\infty} \int_{\gamma_t} \left| \dfrac{\partial\Phi}{\partial s} \right|^2 \mathrm{d}s\mathrm{d}t \leqslant C(\min\limits_{\gamma}\kappa)\varepsilon.
	\end{equation}
	
	Since $\gamma_t$ is star-shaped, we regard it as a graph over $\mathbb{S}^1$ and denote $\gamma_t=\{(\rho(\theta),\theta):\theta\in\mathbb{S}^1\}$. Since by Propositions \ref{prop-curvature} and \ref{prop-c0est}, along the flow (\ref{equ-flow}), we have
	\begin{equation}
		\left| \dfrac{\phi'-u}{\kappa-1} -u \right|\leqslant C (\rho_0,\min\limits_{\gamma}\kappa).
	\end{equation}
	We can characterise the Hausdorff distance from $\gamma$ to $\gamma_t$ by
	\begin{eqnarray}
		\mathrm{dist}(\gamma,\gamma_t) \leqslant \int_0^{t} \left|\partial_\tau X(x,\tau) \right| \mathrm{d}\tau = \int_0^t \left| \dfrac{\phi'-u}{\kappa-1}-u \right|(x,\tau) \mathrm{d}\tau \leqslant C(\rho_0,\min\limits_{\gamma}\kappa)t.
	\end{eqnarray}
	Since
	\begin{equation}
		\int_0^{\sqrt{\varepsilon}} \int_{\gamma_t} \left| \dfrac{\partial\Phi}{\partial s} \right|^2 \mathrm{d}s\mathrm{d}t \leqslant C(\min\limits_{\gamma}\kappa)\varepsilon,
	\end{equation}
	there exists  $t_{\varepsilon}\in(0,\sqrt{\varepsilon}]$ such that
	\begin{equation}
		\int_{\gamma_{t_{\varepsilon}}} \left| \dfrac{\partial\Phi}{\partial s} \right|^2 \mathrm{d}s \leqslant C(\min\limits_{\gamma}\kappa)\sqrt{\varepsilon}. \label{equ-distgamma}
	\end{equation}
	Meanwhile,
	\begin{equation}
		\mathrm{dist}(\gamma,\gamma_{t_{\varepsilon}})\leqslant C(\rho_0,\min\limits_{\gamma}\kappa)\sqrt{\varepsilon}.\label{equ-distest1}
	\end{equation}
	
	We next show that there exists a geodesic circle $\gamma_{\mathbb{H}}$ such that $\mathrm{dist}(\gamma,\gamma_{\mathbb{H}})\leqslant C\varepsilon^{\frac{1}{6}}$. Consider the strictly h-convex curve $\gamma_{t_{\varepsilon}}=\{(\rho(\theta),\theta):\theta\in\mathbb{S}^1\}$ as a graph on $\mathbb{S}^1$. Note that the arc-length element is $\mathrm{d}s= \sqrt{\phi^2(\rho)+\rho_{\theta}^2}\mathrm{d}\theta$, from (\ref{equ-distgamma}) we have,
	\begin{eqnarray}
		C(\min\limits_{\gamma}\kappa)\sqrt{\varepsilon}&\geqslant& \int_{\gamma_{t_{\varepsilon}}} \left| \dfrac{\partial\Phi}{\partial s} \right|^2 \mathrm{d}s = \int_{\mathbb{S}^1} \dfrac{\phi^2(\rho)\rho_{\theta}^2}{\sqrt{\phi^2(\rho)+\rho_{\theta}^2}} \mathrm{d}\theta \notag\\
		&\geqslant& C(\max\limits_{\gamma}|\rho_{\theta}|,\rho_0) \int_{\mathbb{S}^1} \rho_{\theta}^2 \mathrm{d}\theta \notag\\
		&\geqslant& C(\max\limits_{\gamma}|\rho_{\theta}|,\rho_0) \int_{\mathbb{S}^1} \left| \rho-a \right|^2 \mathrm{d}\theta, \label{equ-e}
	\end{eqnarray}
	where we used Poincar\'e inequality in (\ref{equ-e}), and $a=\dfrac{1}{2\pi}\displaystyle\int_{\mathbb{S}^1} \rho(\theta) \mathrm{d}\theta$. This implies that
	\begin{equation}
		\int_{\mathbb{S}^1} |\rho-a|^2 \mathrm{d}\theta \leqslant C_0(\rho_0,\max\limits_{\gamma}|\rho_{\theta}|,\min\limits_{\gamma}\kappa)\sqrt{\varepsilon}.
	\end{equation}
	We claim that: There exists a constant $C_1(\rho_0,\max\limits_{\gamma}|\rho_{\theta}|,\min\limits_{\gamma}\kappa)$ such that
	\begin{equation}
		\max\left\lbrace \max\limits_{\mathbb{S}^1} \rho - a, a-\min\limits_{\mathbb{S}^1} \rho \right\rbrace \leqslant C_1\varepsilon^{\frac{1}{6}}=:L.
	\end{equation}
	Suppose it is not true, without loss of generality, we assume that
	\begin{equation}
		\max\limits_{\mathbb{S}^1} \rho -a =\rho(\theta_0)-a>L.
	\end{equation}
	Then for $l:=\dfrac{L}{2\max\limits_{\mathbb{S}^1}|\rho_{\theta}|}>0$, there holds
	\begin{equation}
		|\rho(\xi)-a|>\dfrac{L}{2},\ \forall\ \xi\in\overline{B}_l(\theta_0) \subset \mathbb{S}^1.
	\end{equation}
	Thus,
	\begin{equation}
		C_0\sqrt{\varepsilon} \geqslant \int_{\mathbb{S}^1} |\rho-a|^2 \mathrm{d}\theta \geqslant \int_{\overline{B}_l(\theta_0)} |\rho-a|^2 \mathrm{d}\theta \geqslant CL^{3}= CC_{1}^3\sqrt{\varepsilon},
	\end{equation}
	which makes a contradiction if we choose $C_1= \left(\frac{2 C_0}{C}\right)^{\frac{1}{3}}$. Therefore, there exists a geodesic circle $\gamma_{\mathbb{H}}$ with radius $a$ such that
	\begin{equation}
		\mathrm{dist}(\gamma_{t_{\varepsilon}},\gamma_{\mathbb{H}}) \leqslant C\|\rho-a\|_{L^{\infty}(\mathbb{S}^1)} \leqslant C(\rho_0,\max\limits_{\gamma}|\rho_{\theta}|,\min\limits_{\gamma}\kappa)\varepsilon^{\frac{1}{6}}.\label{equ-distest2}
	\end{equation}
	
	Combining (\ref{equ-distest1}) with (\ref{equ-distest2}), we conclude that
	\begin{eqnarray}
		\mathrm{dist}(\gamma,\gamma_{\mathbb{H}}) &\leqslant& \mathrm{dist}(\gamma,\gamma_{t_{\varepsilon}}) + \mathrm{dist}(\gamma_{t_{\varepsilon}},\gamma_{\mathbb{H}})\notag\\
		&\leqslant & C(\rho_0,\max\limits_{\gamma}|\rho_{\theta}|,\min\limits_{\gamma}\kappa)\left( \varepsilon^{\frac{1}{2}}+\varepsilon^{\frac{1}{6}} \right),
	\end{eqnarray}
	which is (\ref{equ-curvestability}).
\end{proof}

\subsection{Stability of inequalities for hypersurfaces}
We will apply the nearly umbilical result (Theorem \ref{thm-stability}) to discuss the stability of inequalities (\ref{equ-modifiedquer}) and (\ref{equ-inequality1}).

\begin{proof}[Proof of Theorem \ref{thm-stabilityhypersurface}]
	For (\ref{equ-stabi}), we set
	\begin{equation}
		\varepsilon := \int_{M} (\Phi-u) E_k(\tilde{\kappa}) \mathrm{d}\mu + (n-2k)\tilde{W}_k(\Omega) - \tilde{h}_k\circ\tilde{f}_k^{-1}\left( \tilde{W}_k(\Omega) \right). \label{equ-eps}
	\end{equation}
	Set the center of the interior ball of $\Omega$ with radius $\rho_{-}(\Omega)$ denoted by $o$ as the origin. We evolve $M$ along the flow
	\begin{equation}
		\partial_t X = \left(\dfrac{(\phi'-u)E_{k-1}(\tilde{\kappa})}{E_k(\tilde{\kappa})}-u\right)\nu
	\end{equation}
	and obtain a family of strictly h-convex hypersurfaces $\{M_t\}_{t\in [0,\infty)}$. Denote $\Omega_t$ as the domain enclosed by $M_t$. By Theorem \ref{thm-higherdimension}, $M_t$ is strictly h-convex and $M_t$ converges smoothly to a geodesic sphere $M_{\infty}=\partial\Omega_{\infty}$ centered at the origin as $t\to\infty$. From \cite[Theorem 1]{Borisenko1999TotalCO}, we know
	\begin{equation}
		\rho_{-}(\Omega)\leqslant r(x,t)\leqslant \rho_{-}(\Omega)+\ln 2,\ \forall\ (x,t)\in M_t.
	\end{equation}
	Thus there exists a constant $C=C(\rho_{-}(\Omega))>0$ such that $\Phi(r)\geqslant C>0$. Denote
	\begin{equation}
		Q(t) = \int_{M_t} (\Phi-u)E_k(\tilde{\kappa}) \mathrm{d}\mu_t + (n-2k)\tilde{W}_k(\Omega_t),
	\end{equation}
	then by (\ref{equ-moPhiu}) we have
	\begin{eqnarray}
		\dfrac{\partial}{\partial t} Q(t) &=& (n-k)\int_{M_t} (\Phi-C) E_{k+1}(\tilde{\kappa}) \left( \dfrac{(\phi'-u)E_{k-1}(\tilde{\kappa})}{E_k(\tilde{\kappa})} - u \right) \mathrm{d}\mu_t \label{equ-Phic1}\\
		&&+(n-k)C \int_{M_t} E_{k+1}(\tilde{\kappa}) \left( \dfrac{(\phi'-u)E_{k-1}(\tilde{\kappa})}{E_k(\tilde{\kappa})} - u \right) \mathrm{d}\mu_t \\
		&&-(k+1) \int_{M_t} (\phi'-u) E_{k}(\tilde{\kappa}) \left( \dfrac{(\phi'-u)E_{k-1}(\tilde{\kappa})}{E_k(\tilde{\kappa})} - u\right) \mathrm{d}\mu_t \label{equ-Phic2}\\
		&\leqslant& (n-k)C \int_{M_t} E_{k+1}(\tilde{\kappa}) \left( \dfrac{(\phi'-u)E_{k-1}(\tilde{\kappa})}{E_k(\tilde{\kappa})} - u \right) \mathrm{d}\mu_t\\
		&=& (n-k)C \int_{M_t} (\phi'-u)\left[ \dfrac{E_{k+1}(\tilde{\kappa})E_{k-1}(\tilde{\kappa})}{E_k(\tilde{\kappa})} - E_k(\tilde{\kappa}) \right] \mathrm{d}\mu_t. \label{equ-Phi4}
	\end{eqnarray}
	Here $(\ref{equ-Phic1}) + (\ref{equ-Phic2})\leqslant 0$ for the same reason as (\ref{equ-Mo1}) and (\ref{equ-Mo2}). We also used Minkowski type fomulas (\ref{equ-shiftedminkowski}) in (\ref{equ-Phi4}). Integrating (\ref{equ-Phi4}) on $[0,\infty)$, and combining the fact
	\begin{equation}
		\phi'-u\geqslant \cosh r-\sinh r\geqslant \dfrac{1}{2}\mathrm{e}^{-\rho_{-}(\Omega)}:=C(\rho_{-}(\Omega)), \label{equ-phi'-u}
	\end{equation} 
    we have
	\begin{equation}
		\int_0^{\infty} \dfrac{\partial}{\partial t} Q(t) \mathrm{d}t \leqslant (n-k)C \int_0^{\infty} \int_{M_t}  \dfrac{E_{k+1}(\tilde{\kappa})E_{k-1}(\tilde{\kappa})}{E_k(\tilde{\kappa})} - E_k(\tilde{\kappa})  \mathrm{d}\mu_t \mathrm{d}t. \label{equ-partialtq}
	\end{equation}
	Besides,
	\begin{eqnarray}
		\int_0^{\infty} \dfrac{\partial}{\partial t} Q(t) \mathrm{d}t &=& Q(\infty)-Q(0)\notag\\
		&=& \tilde{h}_k \circ \tilde{f}_k^{-1}\left( \tilde{W}_k(\Omega_{\infty}) \right) - Q(0) \notag\\
		&=& \tilde{h}_k \circ \tilde{f}_k^{-1}\left( \tilde{W}_k(\Omega) \right) - Q(0) \label{equ-tildef} \\
		&=& -\varepsilon, \label{equ-varep}
	\end{eqnarray}
	where in (\ref{equ-tildef}) we used that $\tilde{W}_k(\Omega_t)$ preserves along the flow (\ref{equ-flowhy}), and in (\ref{equ-varep}) we used (\ref{equ-eps}). Combining (\ref{equ-partialtq}) and (\ref{equ-varep}), we get
	\begin{equation}
		\int_0^{\infty} \int_{M_t} E_k(\tilde{\kappa})-\dfrac{E_{k+1}(\tilde{\kappa})E_{k-1}(\tilde{\kappa})}{E_k(\tilde{\kappa})}   \mathrm{d}\mu_t \mathrm{d}t \leqslant \dfrac{C(\rho_{-}(\Omega))}{n-k}\varepsilon. \label{equ-integralepsilon1}
	\end{equation}
	
	For (\ref{equ-stabi2}), we also set
	\begin{equation}
		\varepsilon := \tilde{W}_{k+1}(\Omega) - \tilde{f}_{k+1}\circ \tilde{f}_k^{-1} (\tilde{W}_k(\Omega)).
	\end{equation}
    Note that along the flow \eqref{equ-flowhy}, we have
    \begin{eqnarray}
    	\dfrac{\partial}{\partial t}\tilde{W}_{k+1}(\Omega_t) &=& \int_{M_t} \left[ (\phi'-u) \dfrac{E_{k-1}(\tilde{\kappa})}{E_k(\tilde{\kappa})} - u \right] E_{k+1}(\tilde{\kappa}) \mathrm{d}\mu_t \notag\\
    	&=& \int_{M_t} (\phi'-u) \left[ \dfrac{E_{k-1}(\tilde{\kappa})E_{k+1}(\tilde{\kappa})}{E_{k}(\tilde{\kappa})} - E_k(\tilde{\kappa}) \right] \mathrm{d}\mu_t.
    \end{eqnarray}
    Then we obtain that
    \begin{eqnarray}
    	\int_0^{\infty} \int_{M_t} (\phi'-u) \left[ \dfrac{E_{k-1}(\tilde{\kappa})E_{k+1}(\tilde{\kappa})}{E_{k}(\tilde{\kappa})} - E_k(\tilde{\kappa}) \right] \mathrm{d}\mu_t \mathrm{d}t &=& -\varepsilon,
    \end{eqnarray}
    and (\ref{equ-integralepsilon1}) also holds by using \eqref{equ-phi'-u}.
	
	Denote $F=\dfrac{E_k(\tilde{\kappa})}{E_{k-1}(\tilde{\kappa})}$. It has been verified in \cite[Lemma 7.1, Lemma 7.2 and Proposition 7.4]{Locallyconstrained} that
	\begin{equation}
		0<C_1\leqslant F \leqslant C_2,\ \tilde{\kappa_i}\leqslant C_3,\ i=1,2,\cdots,n,
	\end{equation}
	where the constants $C_1,\ C_2,\ C_3$ depend only on $M$. Thus
	\begin{equation}
		\left| \dfrac{\phi'-u}{F} - u \right| \leqslant C(\rho_{-}(\Omega),M),
	\end{equation}
	and the Hausdorff distance between $\partial\Omega$ and $M_t$ can be estimated by
	\begin{eqnarray}
		\mathrm{dist}(\partial\Omega,M_t) &\leqslant& \mathrm{dist}((X(x,0),X(x,t)) \notag\\
		&\leqslant& \int_0^t \left|\partial_{\tau}X(x,\tau) \right|\mathrm{d}\tau = \int_0^t \left| \dfrac{\phi'-u}{F} - u \right|(x,\tau) \mathrm{d}\tau\notag\\
		&\leqslant & C(\rho_{-}(\Omega),M)t.
	\end{eqnarray}
	
	Note that $S^i_j=h^i_j-\delta^i_j$. After a direct computation, we have
	\begin{equation}
		\mathring{S}^i_j = S^i_j - \dfrac{H-n}{n}\delta^i_j = h^i_j - \dfrac{H}{n} \delta^i_j = \mathring{h}^i_j.
	\end{equation}
	Thus $|\mathring{S}|^2=|\mathring{A}|^2$, and the following inequality holds 
	\begin{equation}
		C(n)E_{k+1,n1}^2(\tilde{\kappa})| \mathring{A} |^2 \leqslant E_k(\tilde{\kappa})^2 - E_{k-1}(\tilde{\kappa}) E_{k+1}(\tilde{\kappa}),
	\end{equation}
	where $\tilde{\kappa}_1\leqslant\tilde{\kappa}_2\leqslant\cdots\leqslant\tilde{\kappa}_n$, $\tilde{\kappa}\in\Gamma_k^+$, $E^2_{k+1,n1}(\tilde{\kappa})=\dfrac{\partial^2 E_{k+1}(\tilde{\kappa})}{\partial\tilde{\kappa_1}\partial\tilde{\kappa_n}}>0$. We refer to \cite[Lemma 3.4]{chenguanlischeuer2022} and \cite[Lemma 4.2]{scheuer2023stability} for more details. Now
	\begin{eqnarray}
		&&\int_0^{\sqrt{\varepsilon}}\int_{M_t} |\mathring{A}|^2 \mathrm{d}\mu_t\mathrm{d}t \notag \\
		&=& \int_0^{\sqrt{\varepsilon}}\int_{M_t} \dfrac{E_{k+1,n1}^2(\tilde{\kappa})|\mathring{A}|^2}{E_k(\tilde{\kappa})}\dfrac{E_k(\tilde{\kappa})}{E^2_{k+1,n1}(\tilde{\kappa})}\mathrm{d}\mu_t\mathrm{d}t \notag\\
		&\leqslant& \dfrac{\max\limits_{t\in[0,\sqrt{\varepsilon}]}E_k(\tilde{\kappa})}{\min\limits_{t\in[0,\sqrt{\varepsilon}]}E_{k+1,n1}^2(\tilde{\kappa})} \int_0^{\sqrt{\varepsilon}}\int_{M_t} \dfrac{E^2_{k+1,n1}(\tilde{\kappa})|\mathring{A}|^2}{E_k(\tilde{\kappa})} \mathrm{d}\mu_t\mathrm{d}t \notag\\
		&\leqslant& C(n)\dfrac{\max\limits_{t\in[0,\sqrt{\varepsilon}]}E_k(\tilde{\kappa})}{\min\limits_{t\in[0,\sqrt{\varepsilon}]}E_{k+1,n1}^2(\tilde{\kappa})} 
		\int_0^{\sqrt{\varepsilon}}\int_{M_t} \dfrac{E_k(\tilde{\kappa})^2 - E_{k-1}(\tilde{\kappa}) E_{k+1}(\tilde{\kappa})}{E_k(\tilde{\kappa})} \mathrm{d}\mu_t \mathrm{d}t \\
		&\leqslant& C(M,\mathrm{dist}(\tilde{\kappa},\partial\Gamma_k^+),\rho_{-}(\Omega),n,k)\varepsilon,
	\end{eqnarray}
	where we used (\ref{equ-integralepsilon1}) in the last inequality. There exists  $t_{\varepsilon}\in(0,\sqrt{\varepsilon}]$ such that
	\begin{equation}
		\int_{M_{t_{\varepsilon}}} |\mathring{A}|^2 \mathrm{d}\mu_{t_{\varepsilon}} \leqslant C(M,\mathrm{dist}(\tilde{\kappa},\partial\Gamma_k^+),\rho_{-}(\Omega),n,k)\sqrt{\varepsilon}. \label{equ-integralepsilon}
	\end{equation} 
	At the same time,
	\begin{equation}
		\mathrm{dist}(\partial\Omega,M_{t_{\varepsilon}}) \leqslant C(\rho_{-}(\Omega),M,n,k)\sqrt{\varepsilon}.  \label{equ-dist1}
	\end{equation}
	
	In the following, we will use Theorem \ref{thm-stability} to prove that there exists a geodesic sphere  $S_{\mathbb{H}}$ such that $\mathrm{dist}(M_{t_{\varepsilon}},S_{\mathbb{H}})\leqslant C(M,\mathrm{dist}(\tilde{\kappa},\partial\Gamma_k^+),\rho_{-}(\Omega),n,k)\varepsilon^{\frac{1}{4}}$. We now follow \cite{Sahjwani2023StabilityOT} and contain the complete proof here. We view hyperbolic space  $\mathbb{H}^{n+1}$ as a ball of radius 2 denoted by $B_2(o)$ in Euclidean space $\mathbb{R}^{n+1}$, and set
	\begin{equation}
		r = \log(2+\rho) - \log(2-\rho).
	\end{equation}
	We have $\mathrm{d}r=\dfrac{4}{4-\rho^2}\mathrm{d}\rho$, $\sinh r=\dfrac{4\rho}{4-\rho^2}$. Let $\mathrm{e}^{\phi(\rho)}=\dfrac{4}{4-\rho^2}$, then
	\begin{equation}
		\overline{g}=\mathrm{d}r^2+\sinh^2 r g_{\mathbb{S}^n} = \mathrm{e}^{2\phi(\rho)}\left( \mathrm{d}\rho^2+\rho^2 g_{\mathbb{S}^n} \right) = \mathrm{e}^{2\phi(\rho)}\tilde{g}.
	\end{equation}
	Therefore, hypersurface $M_{t_{\varepsilon}}$ in $\mathbb{H}^{n+1}$ is conformally as a hypersurface $\tilde{M}_{t_{\varepsilon}}$ in $B_2(o)$. Since $M_{t_{\varepsilon}}$ is h-convex, then $\tilde{M}_{t_{\varepsilon}}$ is convex. Indeed, denote $\tilde{h}^i_j$ as the Weingarten matrix of $\tilde{M}_{t_{\varepsilon}}$, and let $\tilde{\nu}$ be its unit outward normal vector, then there holds (see \cite[Proposition 1.1.11]{curvatureproblem})
	\begin{equation}
		\mathrm{e}^{\phi(\rho)}h^i_j = \tilde{h}^i_j + \mathrm{d}\phi(\tilde{\nu})\delta^i_j.
	\end{equation}
	Note that $\mathrm{d}\phi=\dfrac{2\rho}{4-\rho^2}\mathrm{d}\rho$, we have
	\begin{eqnarray}
		\tilde{h}^i_j &=& \mathrm{e}^{\phi(\rho)} h^i_j- \mathrm{d}\phi(\tilde{\nu})\delta^i_j\notag\\
		&\geqslant& \dfrac{4}{4-\rho^2}h^i_j - \dfrac{2\rho}{4-\rho^2}\delta^i_j\notag\\
		((h^i_j-\delta^i_j)\geqslant 0) &\geqslant& \dfrac{4-2\rho}{4-\rho^2}\delta^i_j \notag\\
		&=& \dfrac{2}{2+\rho}\delta^i_j>0. 
	\end{eqnarray}
	We rescale $\tilde{M}_{t_{\varepsilon}}$ to be $\hat{M}_{t_{\varepsilon}}$ such that $|\hat{M}_{t_{\varepsilon}}|=|\mathbb{S}^n|$. Indeed, it suffices to set
	\begin{equation}
		\hat{M}_{t_{\varepsilon}}=\left( \dfrac{|\mathbb{S}^n|}{|\tilde{M}_{t_{\varepsilon}}|} \right)^{\frac{1}{n}}\tilde{M}_{t_{\varepsilon}}:=\gamma(n,\rho_{-}(\Omega))\tilde{M}_{t_{\varepsilon}}.
	\end{equation}
	Then $\hat{M}_{t_{\varepsilon}}$ is a convex hypersurface in $B_2(o)\subset\mathbb{R}^{n+1}$. Denote $\hat{g}$, $\mathring{\hat{A}}$ as the metric and trace free Weingarten matrix of $\hat{M}_{t_{\varepsilon}}$, then we have
	\begin{eqnarray}
		\mathring{\hat{A}} = \dfrac{1}{\gamma}\mathring{\tilde{A}},\ \ \mathrm{d}\mu_{\hat{g}}= \gamma^n \mathrm{d}\mu_{\tilde{g}},
	\end{eqnarray}
	\begin{equation}
		\mathring{\tilde{A}}=\mathrm{e}^{\phi(\rho)}\mathring{A},\ \ \mathrm{d}\mu_{\tilde{g}} = \mathrm{e}^{-n\phi(\rho)}\mathrm{d}\mu_g.
	\end{equation}
	Therefore,
	\begin{eqnarray}
		\|\mathring{\hat{A}}\|_{L^2(\hat{M}_{t_{\varepsilon}})} = \gamma^{\frac{n-2}{2}} \|\mathring{\tilde{A}}\|_{L^2(\tilde{M}_{t_{\varepsilon}})}, \label{equ-hattraceA}
	\end{eqnarray}
	\begin{eqnarray}
		\|\mathring{\tilde{A}}\|_{L^2(\tilde{M}_{t_{\varepsilon}})} = \left( \int_{M_{t_{\varepsilon}}} \mathrm{e}^{(2-n)\phi(\rho)} |\mathring{A}|^2 \mathrm{d}\mu_g \right)^{\frac{1}{2}} \leqslant C(n,\rho_{-}(\Omega)) \|\mathring{A}\|_{L^2(M_{t_{\varepsilon}})}.  \label{equ-tildetraceA}
	\end{eqnarray}
	Combining (\ref{equ-integralepsilon}), (\ref{equ-hattraceA}) and (\ref{equ-tildetraceA}), we obtain that
	\begin{equation}
		\|\mathring{\hat{A}}\|_{L^2(\hat{M}_{t_{\varepsilon}})} \leqslant  C(M,\mathrm{dist}(\tilde{\kappa},\partial\Gamma_k^+),\rho_{-}(\Omega),n,k)\varepsilon^{\frac{1}{4}}.
	\end{equation}
	We now set $p=2$ in Theorem \ref{thm-stability}, then there exist a parametrization $\varphi:\mathbb{S}^n\to \hat{M}_{t_{\varepsilon}}\subset B_2(o)$ and a point $\mathcal{P}\in B_2(o)$ such that
	\begin{eqnarray}
		\|\varphi-\mathrm{id}-\mathcal{P}\|_{W^{2,2}(\mathbb{S}^n)} \leqslant C(n)\|\mathring{\hat{A}}\|_{L^2(\hat{M}_{t_{\varepsilon}})} \leqslant  C(M,\mathrm{dist}(\tilde{\kappa},\partial\Gamma_k^+),\rho_{-}(\Omega),n,k)\varepsilon^{\frac{1}{4}}.
	\end{eqnarray} 
	This implies that $\hat{M}_{t_{\varepsilon}}$ is Hausdorff-close to some unit sphere, thus $\tilde{M}_{t_{\varepsilon}}$ is Hausdorff-close to a sphere with radius $\gamma^{-1}$. Therefore, there exists a geodesic sphere $S_{\mathbb{H}}$ such that
	\begin{equation}
		\mathrm{dist}(M_{t_{\varepsilon}},S_{\mathbb{H}}) \leqslant C(M,\mathrm{dist}(\tilde{\kappa},\partial\Gamma_k^+),\rho_{-}(\Omega),n,k)\varepsilon^{\frac{1}{4}}. \label{equ-dist2}
	\end{equation} 
	
	Combining (\ref{equ-dist1}) and (\ref{equ-dist2}), we obtain that
	\begin{eqnarray*}
		\mathrm{dist}(\partial\Omega,S_{\mathbb{H}}) &\leqslant& \mathrm{dist}(\partial\Omega,M_{t_{\varepsilon}})+\mathrm{dist}(M_{t_{\varepsilon}},S_{\mathbb{H}})\\
		&\leqslant& C(M,\mathrm{dist}(\tilde{\kappa},\partial\Gamma_k^+),\rho_{-}(\Omega),n,k)\left(\varepsilon^{\frac{1}{2}}+\varepsilon^{\frac{1}{4}}\right).
	\end{eqnarray*}
	These are (\ref{equ-stabi2}) and (\ref{equ-stabi}).
\end{proof}

%\begin{ack}
%The authors would like to thank Professor Yong Wei for his helpful discussions and constant support. The authors were supported by National Key Research and Development Program of China 2021YFA1001800.
%\end{ack}
%----------------------------------------------------------

\bibliographystyle{plainnat}
\bibliography{reference.bib}

\end{document}